\documentclass[10pt, reqno]{amsart}
\usepackage{amssymb,amsfonts}
\usepackage{amsbsy}
\usepackage{amsfonts}
\usepackage{amsmath}
\usepackage{amsthm}
\usepackage{amssymb}
\usepackage{ytableau}
\PassOptionsToPackage{dvipsnames}{xcolor}
\usepackage{pgfplots}
\pgfplotsset{compat=1.11}
\usepackage{amsrefs}
\usepackage[all,arc]{xy}
\usepackage{tikz-cd}
\usepackage{enumerate}
\usepackage{mathrsfs}

\theoremstyle{plain} 
\newtheorem{thm}{Theorem}[section]
\newtheorem{cor}[thm]{Corollary}

\newtheorem{lem}[thm]{Lemma}

\theoremstyle{definition}
\newtheorem{defn}[thm]{Definition}

\theoremstyle{remark}
\newtheorem{rem}[thm]{Remark}

\numberwithin{equation}{section}

\newcommand{\tR}{{\widetilde{R}}} 
\newcommand{\G}{{\overline{G}}}

\newcommand{\ord}{{\rm ord}}

\newcommand{\EE}{{\mathbb E}}

\definecolor {UMblue}  {RGB}{0, 39, 76}
\definecolor {UMmaize} {RGB}{255, 203, 5}

\definecolor {color_b}{RGB}{255,0,0}
\definecolor {color_c}{RGB}{20, 200, 30}
\definecolor {color_a}{RGB}{0,0,255}

\definecolor{lgreen} {RGB}{180,210,100}
\definecolor{dblue}  {RGB}{20,66,129}
\definecolor{ddblue} {RGB}{11,36,69}
\definecolor{lred}   {RGB}{220,0,0}
\definecolor{nred}   {RGB}{224,0,0}
\definecolor{norange}{RGB}{230,120,20}
\definecolor{nyellow}{RGB}{255,221,0}
\definecolor{ngreen} {RGB}{98,158,31}
\definecolor{dgreen} {RGB}{78,138,21}
\definecolor{nblue}  {RGB}{28,130,185}
\definecolor{jblue}  {RGB}{20,50,100}

\title{Partial Factorizations of Products of Binomial Coefficients}

\author{Lara Du and  Jeffrey C. Lagarias}
\address{Dept. of Mathematics, University of Michigan, Ann Arbor, MI 48109-1043, USA.} 
\date{September 17, 2020}
\begin{document}

\begin{abstract}
Let $\G_n= \prod_{k=0}^n \binom{n}{k},$
the product of the elements of the $n$-th row of Pascal's triangle.
This paper studies the partial factorizations
of $\G_n$ given by the product  $G(n,x)$ of all prime factors $p$ of $\G_n$ having $p \le x$,
counted with multiplicity. It shows $\log G(n, \alpha n) \sim f_G(\alpha)n^2$ as $n \to \infty$ for a limit function $f_{G}(\alpha)$
defined for $0 \le \alpha \le 1$. 
The main results are deduced from study of functions  $A(n, x), B(n,x),$
that encode statistics of the  base $p$ radix  expansions of the  integer $n$ (and smaller integers), 
where the base $p$ ranges over  primes $p \le x$.  
Asymptotics of $A(n,x)$ and $B(n,x)$  are derived  using the prime number theorem with remainder term or conditionally on  the Riemann hypothesis.

\end{abstract}

\maketitle

\vspace{-2em}

\vspace{1cm}
\section{Introduction}

 Let $\overline{G}_n$ denote  the product of the binomial coefficients in the $n$th row of Pascal's triangle
\begin{equation}\label{eqn:defGn}
\overline{G}_n:=\prod_{k=0}^{n}\binom{n}{k} = \frac{ (n!)^{n+1}} { \prod_{k=0}^{n} (k!)^2}.
\end{equation}
 These products were studied in \cite{LagM:2016}, 
where it was observed that the integer sequence $\G_n$ arises
 as the inverse of the product of all the non-zero unreduced Farey fractions,
  i.e. the set of all rational fractions in the unit interval $(0,1]$ having denominator at most $n$,
not necessarily  in lowest terms.
We write the prime factorization of $\G_n$ as 
 \begin{equation}\label{eqn:productformula}
 \G_n = \prod_p p^{\nu_p(\G_n)}
 \end{equation} 
 where $\nu_p(\G_n) = \ord_p(\G_n)$. 
Since $\G_n$ is an integer,  $\nu_p(\G_n) \ge 0$ for all $n \ge 1$.
The  asymptotic growth rate of $\G_n$  is easily determined,
using Stirling's formula, to be 
\begin{equation}\label{eqn:logG-asymp}
\log \overline{G}_n = \frac{1}{2}n^2 
- \frac{1}{2} n \log n + O(n),
\end{equation}
an estimate  which  is valid more generally for 
  the step function $\overline{G}_x := \overline{G}_{\lfloor x \rfloor}$ for all real $x \ge 1$. 
  The sequence  $\G_n$ considered only  at integer points $n$  has  
a complete asymptotic expansion for $\log \overline{G}_n$ to all orders in  $(\frac{1}{n})^k$ $(k \ge 0)$,
see \cite[Theorem A.2]{LagM:2016}.)

The purpose of this paper is to study  the internal  structure of the prime factorization of $\G_n$
as $n$ varies, as measured by the partial factorization
\begin{equation}\label{eqn:Gnx-def}
G(n, x) = \prod_{p \le x} p^{\nu_p(\G_n)}.
\end{equation} 
Here $G(n,x)$  is a divisor of $\G_n$  that includes   the total contribution of all primes up to $x$ in the product $\G_n$.
The function $G(n,x)$  for fixed $n$ 
 is an integer-valued step function of the variable $x$.
This function of $x$ stabilizes for $x \ge n$, with 
$$
G(n,x) = G(n,n)= \G_n  \quad \mbox{for} \quad x \ge n. 
$$
This paper determines the asymptotic behavior of $\log G(n,x)$ and related 
arithmetic statistics as $n \to \infty$ for a wide range of $x$, with emphasis on the range
when $x \sim \alpha n $, for fixed $0< \alpha\le 1$. To do so it determines the
asymptotic behavior of auxiliary statistics $A(n,x)$ and $B(n,x)$, defined below,  which encode
information on radix expansions of integers up to $n$ to prime bases $p \le n$. 

%
%
\subsection{Result: Asymptotics of $G(n,x)$} \label{sec:11} 

 We determine the size of the partial factorization
 function $G(n, x)$ in the range $1 \le x \le n$.
We establish limiting behavior as $n \to \infty$ taking  $x= x(n):= \alpha n$. 

%
%
\begin{thm}\label{thm:Gnx-main}
Let $G(n, x) = \prod_{p \le x} p^{\nu_p(\G_n)}$. Then for  all $0 < \alpha \le 1$,  
\begin{equation}\label{eqn:Gnx-main}  
\log G(n, \alpha n) = f_G(\alpha)  n^2 + R_{G}(n, \alpha n), 
\end{equation}
where $f_G(\alpha)$  is a  function given for $\alpha>0$ by
 \begin{equation}\label{eqn:Gnx-parametrized} 
f_G(\alpha) =  \frac{1}{2} 
%
+ \frac{1}{2} \alpha^2 \left\lfloor \frac{1}{\alpha}\right\rfloor^2 + \frac{1}{2} \alpha^2 \left\lfloor \frac{1}{\alpha} \right\rfloor-  \alpha \left\lfloor \frac{1}{\alpha} \right\rfloor,
\end{equation}  
with $f_G(0)=0$ and $R(n, \alpha n)$ is a remainder term. 

(1) Unconditionally there is a positive constant $c$ such that for all $n \ge 4$, and all  $0 < \alpha \le 1$ 
the remainder term satisfies
\begin{equation}\label{eqn:remU}
R_{G}(n, \alpha n) = O \left( \frac{1}{\alpha} n^2 \exp(- c \sqrt{\log n})\right). 
\end{equation} 
The implied constant in the $O$-notation does not depend on  $\alpha$.

(2) Conditionally on the Riemann hypothesis, for  all $n \ge 4$ and  all $0 < \alpha \le 1$, 
the remainder term satisfies
\begin{equation}\label{eqn:remRH}
R_{G}(n, \alpha n) = O \left( \frac{1}{\alpha} n^{7/4} (\log n)^2 \right),
\end{equation}
The  implied constant in the $O$-notation does not depend on  $\alpha$.
\end{thm}  

The limit function $f_G(\alpha)= \lim_{n \to \infty} \frac{1}{n^2} \log G(n, \alpha n)$
is pictured in Figure \ref{fig:AB1}.

%

\begin{figure}[h] 
\begin{center}
\begin{tikzpicture}[scale=0.90]
    \begin{axis}[
        xmin=0,xmax=1,
        ymin=0,ymax=0.5,
        minor tick num=3
      ]
      \draw[blue,dashed] (0,0) -- (1,0.5);
      \draw[line cap=round,red,thick] \foreach \n in {1,...,100}{
        ({1/(\n+1)},{0.5/(\n+1)}) parabola ({1/\n},{0.5/\n})
      } -- (0,0);
    \end{axis}

  \end{tikzpicture}
\end{center}
\caption{Graph of limit function $f_{G}(\alpha)$ in $(\alpha,\beta)$-plane for  $0 \le \alpha \le 1.$
The dotted line is $\beta= \frac{1}{2} \alpha$.}
\label{fig:AB1}
\end{figure}
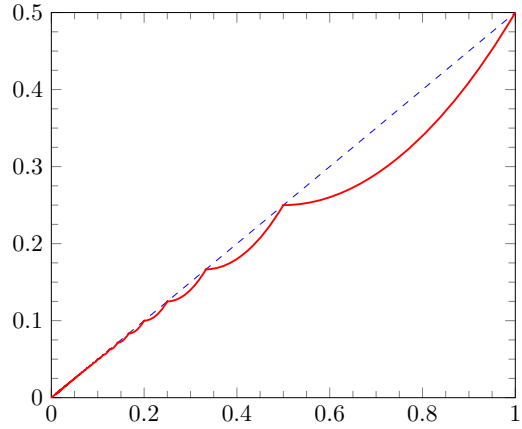 

The  limit function $f_G(\alpha)$
 has the following properties (cf. Lemma \ref{lem:fG}).
\begin{enumerate}
\item[(i)]
The function $f_{G}$  is continuous on $[0,1]$. 
It has    $\lim_{\alpha \to 0^{+}} f_G(\alpha)=0$ and the formula \eqref{eqn:Gnx-parametrized} gives the value $f_{G}(0)=0$, making
  the convention that $\alpha \lfloor \frac{1}{\alpha} \rfloor =1$ at $\alpha=0$.
It  is not differentiable at each point $\alpha= \frac{1}{j}$ for integer $j \ge 2$,
and not differentiable from above at $\alpha=0$.
\item [(ii)]
 The function $f_G$   satisfies
\begin{equation}
f_G( \alpha) \le \frac{1}{2}\alpha  \quad \mbox{for} \quad 0 \le \alpha \le 1.
\end{equation} 
Equality holds   at   $\alpha=\frac{1}{j}$ for all integer  $j \ge 1$,
with  $f_G\left(\frac{1}{j}\right) = \frac{1}{2 j}$, and at $\alpha=0$ (by convention)  and  at no other values. 
\end{enumerate} 
Specifically $f_{G}(\alpha)$ is   piecewise quadratic, i.e. for  $j \ge1$, on each closed interval $\left[ \frac{1}{j+1}, \frac{1}{j}\right]$ it is  given by 
\begin{equation}
f_G(\alpha) = \frac{1}{2} - j\alpha + \frac{j(j+1)}{2}  \alpha^2 \quad \mbox{for} \quad \frac{1}{j+1} \le x \le \frac{1}{j}.
\end{equation}

Theorem \ref{thm:Gnx-main}  is  a restated form of   Theorem \ref{thm:Gnx}, 
which applies uniformly to  the full range $1 \le x \le n$. 
 The new  content of Theorem  \ref{thm:Gnx-main}  concerns the range $0< \alpha <1$
and determination of the function $f_{G}(\alpha)$.  
At the endpoint $\alpha=1$  the binomial product estimate  \eqref{eqn:logG-asymp} gives an unconditional 
asymptotic formula  for $\log G(n,n)$ with  power-savings remainder term { better than} \eqref{eqn:remRH}.
This value $G(n,n)$  is given explicitly by a product of ratios of factorials, permitting the estimate.
Improved estimates are possible in  
some parts of the range $x=o(n)$, corresponding to $\alpha=0$,  
using exponential sum methods, as discussed at the end of Section \ref{sec:12}. 

We now consider the methods used to prove 
Theorem \ref{thm:Gnx-main}. It  is proved starting  from an expression for $\log G(n,x)$  
obtained  taking  the logarithm of  the factorization \eqref{eqn:Gnx-def}: 
 \begin{equation}\label{eqn:productformula2}
 \log G(n,x)  = \sum_{p\le x} \nu_p(\G_n) \log p. 
 \end{equation}
  The proof uses formulas 
  for the individual exponents  $\nu_p(\G_n)$
 given in terms of base $p$ radix expansion data of the integers up to $n$,
 proved in \cite{LagM:2016} 
  and stated in Section \ref{sec:12} below. 
 The  functions $\nu_p(\G_n)$ exhibit a kind of self-similar behavior, different for each $p$,
having large fluctuations.  
One can write the individual exponents $\nu_p(\G_n)$ as a difference of 
quantities given by statistics of the base $p$ radix expansion of integers up to $n$ 
(see Theorem \ref{thm:explicit} ).
Summing over $p \le x $ yields  a formula 
$$
\log G(n,x) = A(n,x) - B(n,x),
$$
involving  nonnegative arithmetic functions   $A(n,x)$ and $B(n,x)$ 
defined in \eqref{eqnx:A-function} and \eqref{eqnx:B-function} below. 
The functions $A(n,x)$ and $B(n,x)$ encode information on  prime number counts, as
detailed in Section \ref{sec:15} below.

The main technical results of this paper are estimates of  the  the size of  $A(n, x)$ and $B(n,x)$.
These functions  are weighted averages of statistics of the radix expansions of $n$ for  varying  prime bases $p \le x$.
Individual radix statistics have  been extensively studied in the literature, holding the radix base $p$ fixed and varying $n$.
The case of fixed $n$ and variable $p$  considered here seems  not well studied.

%
%
\subsection{Products of binomial coefficients and digit sum statistics}\label{sec:12}

   It is well known that the divisibility of binomial coefficients ${{n}\choose{k}}$ by
 prime powers is described by base $p$ radix expansion conditions, starting from work of Kummer, see the survey of  Granville \cite{Gra97}.
     Given a base $b \ge 2$, write the base $b$  radix expansion of an integer $n \ge 0$ as
         $$
   n = \sum_{i=0}^{k}  a_i b^i \quad \mbox{with} \quad 0 \le a_i= a_i(b, n) \le b-1,
   $$
 in which    $b^k \le n < b^{k+1}$ and the top digit  $a_k(b, n) \ge 1$. One has
 \begin{equation}\label{eqn:floor-recursion}
 a_i(b,n) = \left\lfloor \frac{n}{b^i} \right\rfloor - b\left\lfloor \frac{n}{b^{i+1}} \right\rfloor. 
 \end{equation}
The radix conditions involve the following two statistics of the base $b$ digits of $n$.

  \begin{defn}

 (1)  The {\em sum of digits function} $d_b(n)$ (to base $b$) is
   \begin{equation}
   d_b(n) := \sum_{i \ge 0} a_i(b, n).
  \end{equation}
   
  (2)  The {\em running digit sum function} $S_b(n)$ (to base $b$) is
   \begin{equation}
   S_b(n) := \sum_{j=0}^{n-1} d_b(j).
  \end{equation}
   \end{defn}

The paper \cite{LagM:2016} derived a closed  formula  for $\ord_p(\G_n)$ which involves such radix expansions
 of  the integers $1 \le j \le n$.

\begin{thm}\label{thm:explicit} 
{\rm (\cite[Theorem 5.1]{LagM:2016})}
For each prime $p$ one has for each $n \ge 1$, 
  \begin{equation}\label{eqn:localformula}
  \nu_p(\G_n) = \frac{2}{p-1} S_p(n) - \frac{n-1}{p-1} d_p(n).
  \end{equation}
\end{thm}

  The formula  \eqref{eqn:localformula} encodes large cancellations 
 of  powers of $p$ between the numerator and denominator
 of the factorial form for $\G_n$ on the right side of \eqref{eqn:defGn}.
 The individual terms  on the right side of \eqref{eqn:localformula}
 need not be integers: As an example,  $n=35$  has base $p= 7$ expansion $(50)_7$, whence 
 $d_7(35)=5$ and $\frac{n-1}{p-1} d_p(n)= \frac{85}{3}$, while $\frac{2}{p-1}S_p(n)= \frac{175}{3}$
 and $\nu_7(\G_{35}) = 30$.

 Taking  logarithms of both sides of the product formula \eqref{eqn:Gnx-def} 
for $G(n,x) $ and substituting the formula \eqref{eqn:localformula}  
for each $\nu_p(\G_n)$ 
yields  the following identity.
There  holds 
\begin{equation}\label{eqn:GABx}
\log G(n,x)  = A(n, x) - B(n, x),
\end{equation} 
where
\begin{equation}\label{eqnx:A-function}
A(n, x) = \sum_{p \le x} \frac{2}{p-1} S_p(n) \log p 
\end{equation}
and
\begin{equation}\label{eqnx:B-function}
B(n, x) = \sum_{p\leq x}\frac{n-1}{p-1} d_p(n) \log p. 
\end{equation} 
The functions $A(n, x)$ and $B(n, x)$ are arithmetical sums
that combine 
behavior of the base $p$ digits of the integer $n$, viewing $n$ as fixed, and varying the radix base  $p$.
The interesting range of $x$  is  $1 \le x \le n $  because these functions ``freeze" at $x=n$: 
$A(n, x)= A(n,n)$ for $x \ge n$ and $B(n,x) = B(n,n)$ for $x \ge n$.

We single out the  special case $x=n$, setting 
\begin{equation}\label{eqn:A-function}
A(n) := A(n,n)= \sum_{p \le n} \frac{2}{p-1} S_p(n) \log p 
\end{equation}
and
\begin{equation}\label{eqn:B-function}
B(n) :=B(n,n) =  \sum_{p\leq n}\frac{n-1}{p-1} d_p(n) \log p. 
\end{equation}
 The sums $A(n)$ and $B(n)$ hold $n$ fixed and vary the base $p$. 

The main results of the paper estimate the functions $A(n,x), B(n,x)$ and $\log G(n,x)$ for $1 \le x \le n$
 with  main terms having  the general form  $f(\alpha) n^2$ where $\alpha= \frac{x}{n}$ and with  such
  $f(\alpha)$ for $0\le \alpha \le 1$ is a continuous function having $f(0)=0$.
  The proofs  first  estimate $A(n,n)$ and $B(n,n)$,  and then use  these estimates as input  
to  recursively estimate $A(n,x)$ and $B(n,x)$ for general $x$.

Olivier Bordell\`{e}s informed us  that 
 exponential sum  methods yield alternative unconditional estimates 
for $A(n,x)$, $B(n, x)$ and $\log G(n,x)$, which are nontrivial when $x=o(n)$,
and apply for $x > \sqrt{n}$. 
These estimates improve on the estimates of our main
theorems for certain ranges of $x$.  
We present such estimates in  Appendix \ref{sec:appendix1}.
The main terms in the exponential sum  estimates have a different form than the main terms in  the
estimates for $A(n,x)$, $B(n,x)$ and $\log G(n,x)$.  Theorem \ref{thm:ABnx} 
obtains for $x=o(n)$   a simplified form of our main terms 
 which facilitates a comparison of the estimates. 
%
%
\subsection{Results: Asymptotics of $A(n)$ and $B(n)$}\label{sec:13}

We determine  asymptotics of the two functions $A(n)$ and $B(n)$  as $n \to \infty$, giving
a main term and  a bound on the remainder term. The analysis proceeds by estimating the fluctuating term $B(n)$
depending on $d_p(n)$.

\begin{thm}\label{thm:ABn}
Let 
$A(n)  = \sum_{p \le n} \frac{2}{p-1} S_p(n) \log p$
and 
$B(n) = \sum_{p\leq n}\frac{n-1}{p-1} d_p(n) \log p.$ 

(1) There is a constant $c>0$, such that for  $n \ge 4$, 
\begin{equation}\label{eqn:Uncond-A-aysmp} 
A(n) = \left(\frac{3}{2} -\gamma \right) n^2 + O \left( n^2 \exp (-c \sqrt{\log n}) \right),
\end{equation}
where $\gamma$ denotes Euler's constant. Similarly
\begin{equation}\label{eqn:Uncond-B-asymp} 
B(n) = (1 -\gamma) n^2 + O \left( n^2 \exp (-c \sqrt{\log n}) \right).
\end{equation}

(2) Assuming the Riemann hypothesis, for all $n \ge 4$,
\begin{equation} \label{eqn:RH-A-asymp}
A(n) = \left(\frac{3}{2} -\gamma \right) n^2 + O\left( n^{7/4} (\log n)^2\right).
\end{equation}
and 
\begin{equation}\label{eqn:RH-B-asymp} 
B(n) = (1 -\gamma) n^2 + O \left( n^{7/4} (\log n)^2\right),
\end{equation}
\end{thm} 

Theorem \ref{thm:ABn} answers  a question raised in \cite[Section 8]{LagM:2016} of  whether the
asymptotic growth  of $A(n)$ is the same as that of the sum  $A^{\ast}(n)$ obtained by replacing each   $S_p(n)$ with the  leading term of its asymptotic
growth estimate as $n \to \infty$. They are not the same:  see Section \ref{sec:15}.

To establish Theorem \ref{thm:ABn} it suffices to prove it for  $B(n)$;   the estimate for $A(n)$  then follows
from the linear relation $A(n)  = \log \G_n + B(n)$ (from \eqref{eqn:GABx}) combined with the asymptotic
estimate  for $G(n)$ in \eqref{eqn:logG-asymp}.
The main contribution in the sum  $B(n)$ comes from  those  primes  $p$
having  $p > \sqrt{n}$, whose key  property  is that :
{\em their base $p$  radix expansions have  exactly two digits}. 
The size of the remainder term then involves prime counting functions, which relate to the zeros of the Riemann zeta function. 
 We obtain an  unconditional result from the standard zero-free region for $\zeta(s)$.
The Riemann hypothesis, or more generally a zero-free region for the zeta function of the form  $Re(s)> 1- c_0$
for some $c_0>0$ 
yields an  asymptotic formula of shape 
$B(n) = (1-\gamma)n^2 + O\left( n^{2- \delta}\right)$ with a power-saving remainder term
$\delta= \delta(c_0)$ depending on the width of the zero-free region.

 The constants appearing in the main terms of the asymptotics of  $A(n)$ and $B(n)$  in Theorem \ref{thm:ABn}
 give quantitative information on 
cross-correlations between the statistics $d_p(n)$ and $S_p(n)$  
of the base $p$ digits of $n$ (and smaller integers) as the base $p$ varies while $n$ is held fixed. 
As suggested in the survey \cite{Lag:13}, the occurrence of Euler's constant in the main term of these asymptotic estimates 
encodes subtle arithmetic behavior in these sums.

%
%
\subsection{Results: Asymptotics of $A(n,x)$ and $B(n,x)$}\label{sec:14}

We first  determine  asymptotics for  $B(n, \alpha n)$ for $0 \le \alpha \le 1$,
starting  from $B(n)=B(n,n)$ and obtaining $B(n,x)$ by decreasing $x$ from $x=n$.
In what follows $H_m =\sum_{j=1}^m \frac{1}{j}$ denotes the $m$-th harmonic number
and $\gamma$ denotes Euler's constant.

\begin{thm}\label{thm:Bnx-cor}
Let $B(n, \alpha n) = \sum_{p\leq \alpha n}\frac{n-1}{p-1} d_p(n) \log p.$ 
Then for  all $0 < \alpha \le 1$,  
\begin{equation}\label{eqn:Bnx-main}  
B(n, \alpha n) = f_B(\alpha)  n^2 + R_{B}(n, \alpha n), 
\end{equation}
where $f_B(\alpha)$  is a function given for $\alpha>0$ by
 \begin{equation}\label{eqn:Bnx-parametrized} 
f_B(\alpha) =  1- \gamma+  \left( H_{\lfloor \frac{1}{\alpha}\rfloor}- \log \frac{1}{\alpha} \right)  - \alpha \left\lfloor \frac{1}{\alpha}\right\rfloor,
\end{equation} 
with $f_B(0)=0$,
and  $R_{B}(n, \alpha n)$ is a remainder term. 

(1) Unconditionally there is a positive constant $c$ such that for all $n \ge 4$, and $0 < \alpha \le 1$,
the remainder term satisfies
\begin{equation}\label{eqn:BremU}
R_{B}(n, \alpha n) = O \left( \frac{1}{\alpha} n^2 \exp(- c \sqrt{\log n})\right). 
\end{equation} 
The implied
 constant in the $O$-notation does not depend on  $\alpha$.

(2) Conditionally on the Riemann hypothesis, for  all $n \ge 4$ and $0 < \alpha \le 1$, 
the remainder term satisfies
\begin{equation}\label{eqn:BremRH}
R_B(n, \alpha n) = O \left( \frac{1}{\alpha} n^{7/4} (\log n)^2 \right),
\end{equation}
The  implied constant in the $O$-notation does not depend on  $\alpha$.
\end{thm} 

The limit function $f_{B}(\alpha)$  is pictured in Figure \ref{fig:B2}.
The function lies strictly above the diagonal line $\beta=(1-\gamma)\alpha$; note that  in  \eqref{eqn:GABx}
in   its relation to $\log G(n, x)$ it appears with a negative sign, consistent with $f_{G}(\alpha) \le \frac{1}{2} \alpha$.

%

\begin{figure}[h] 
\includegraphics[scale=0.60]{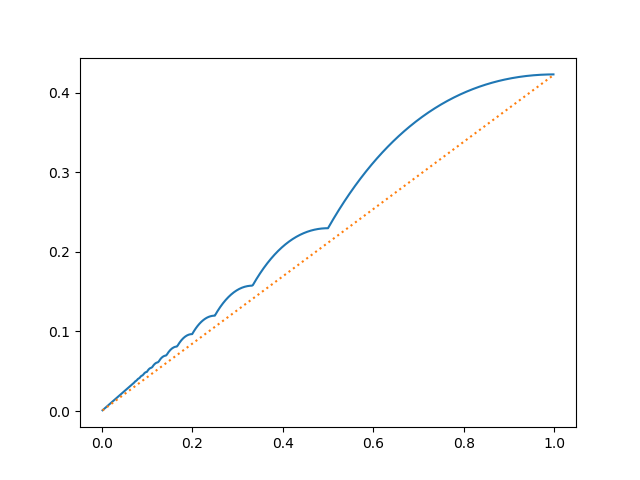}
\caption{Graph of limit function $f_{B}(\alpha)$ in $(\alpha, \beta)$-plane,  $0 \le \alpha \le 1.$ The dotted line is $\beta= (1-\gamma)  \alpha$,
where $\gamma$ is Euler's constant.}
\label{fig:B2}
\end{figure} 

We then obtain asymptotics for $A(n,x)$ using a recursion starting from $A(x,x)$ (given by \eqref{eqn:AB-recursion}) relating $A(n, x)$ to various $B(y,x)$
with $x \le y \le n$.

\begin{thm}\label{thm:Anx-cor}
Let $A(n, \alpha n ) = \sum_{p\leq \alpha n}\frac{2}{p-1} S_p(n) \log p.$ 
Then for  all $0 < \alpha \le 1$, 
\begin{equation}\label{eqn:Anx-main}  
A(n, \alpha n) = f_A(\alpha)  n^2 + R_{A}(n, \alpha n), 
\end{equation}
where $f_A(\alpha)$  is a function given for $\alpha>0$ by
 \begin{equation}\label{eqn:Anx-parametrized} 
 f_A(\alpha) =  \frac{3}{2} - \gamma+  \left( H_{\left\lfloor \frac{1}{\alpha}\right\rfloor}- \log \frac{1}{\alpha} \right)  
%
+ \frac{1}{2} \alpha^2 \left\lfloor \frac{1}{\alpha}\right\rfloor^2 + \frac{1}{2} \alpha^2 \left\lfloor \frac{1}{\alpha} \right\rfloor
- 2 \alpha \,\left\lfloor \frac{1}{\alpha} \right\rfloor,
\end{equation}
with $f_{A}(0) =0$, 
and  $R_{A}(n, \alpha n)$ is a remainder term. 

(1) Unconditionally there is a positive constant $c$ such that for all $n \ge 4$, and $0 < \alpha \le 1$.
the remainder term satisfies 
\begin{equation}\label{eqn:AremU}
R_{A}(n, \alpha n) = O \left( \frac{1}{\alpha} n^2 \exp(- c \sqrt{\log n})\right). 
\end{equation} 
The implied constant in the $O$-notation does not depend on  $\alpha$.

(2) Conditionally on the Riemann hypothesis, for  all $n \ge 4$ and $0 < \alpha \le 1$,
the remainder term satisfies 
\begin{equation}\label{eqn:AremRH}
R_{A}( n, \alpha n) = O \left( \frac{1}{\alpha} n^{7/4} (\log n)^2 \right),
\end{equation}
The  implied constant in the $O$-notation does not depend on  $\alpha$.
\end{thm}  

The limit function $f_{A}(\alpha)$  is pictured in Figure \ref{fig:A3}. It lies
very close to the line $\beta= (3/2 -\gamma) \alpha$. 
The graph of 
$f_{A}(\alpha)$   falls  below the line $\beta= (3/2 -\gamma) \alpha$
for $\alpha> \alpha_0$ and falls above it for $\alpha < \alpha_0$, with $\alpha_0 \approx 0.82$. 
The figure also depicts  a plot of its derivative $f_{A}^{'}(\alpha)$, with horizontal dotted line
indicating  derivative $\frac{3}{2} -\gamma$.
%

\begin{figure}[h] 
\includegraphics[scale=0.60]{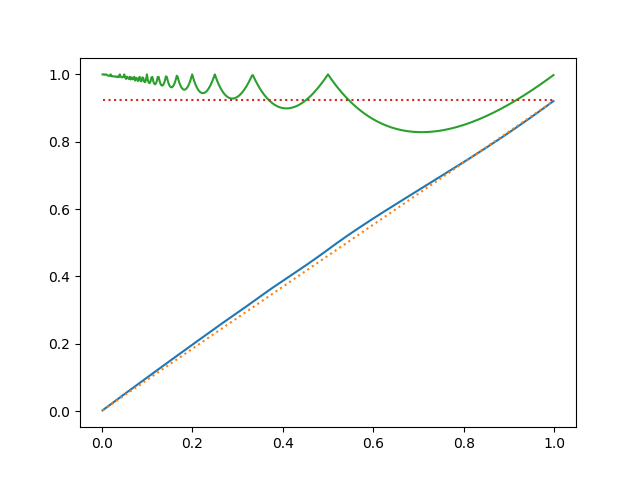}
\caption{Graph of limit function $f_{A}(\alpha)$ in $(\alpha, \beta)$-plane,  $0 \le \alpha \le 1.$ The dotted line is $\beta= (3/2-\gamma)  \alpha$,
where $\gamma$ is Euler's constant. Superimposed on the graph is a plot of  the derivative $f_A^{'}(\alpha)$ drawn to  the same scale.} 
\label{fig:A3}
\end{figure} 

We note that the functions $f_{A}(\alpha)$ and $f_{B}(\alpha)$ 
are continuous functions of $\alpha$, although the given floor function 
 formulas for $f_{A}(\alpha)$ and $f_{B}(\alpha)$ are a sum of functions that are discontinuous 
at the points $\alpha= \frac{1}{k}$.

Theorem \ref{thm:Bnx-cor} and Theorem \ref{thm:Anx-cor} are restated
versions of Theorems \ref{thm:Bnx} and Theorem \ref{thm:Anx} given in
terms of the $x$-variable. Theorem \ref{thm:Gnx-main} follows as a corollary of these 
two theorems, substituting their estimates into the formula
$\log G(n,x)  = A(n, x) - B(n, x)$. 
In the subtraction giving the asymptotics of $\log G(n,x)$, Euler's constant cancels out.

 %
%
\subsection{Motivation: Digit sum statistics and the prime number theorem }\label{sec:15} 

The  statistics $A(n,x)$ and $B(n,x)$ can be related to the problem of estimating $\pi(x)$. 
\subsubsection{Running digit sum $S_b(n)$} 
The radix statistic
 $S_b(n)$ a fixed integer base $b \ge 2$
 has been extensively studied.
 It was treated in 1940 by  Bush \cite{Bush:40}, followed by Bellman and  Shapiro \cite{BelS:48}, and Mirsky \cite{Mir:49}, who in 1949 showed
that for all $b \ge 2$, 
\begin{equation}\label{eqn:mirsky}
S_b(n) = n \log_b(n) +O_{b}(n),
\end{equation}  
where the implied constant in the $O$-notation depends on  the base $b$. 
In 1952  Drazin and Griffith \cite{DG52} 
deduced an inequality implying that for all bases $b \ge 2$, 
\begin{equation}\label{eqn:S-ineq} 
S_b(n) \le \frac{b-1}{2} n \log_b n  \quad \mbox{ for all} \quad n \ge 1, 
\end{equation}
 with equality holding for  $n=b^k$ for $k \ge 1$, cf. \cite[Theorem 5.8]{LagM:2016}. 
The upper bound \eqref{eqn:S-ineq} suggests consideration of the statistic
 \begin{equation}\label{eqn:Anx-star} 
 A^{\ast}(n, x) := \sum_{p \le x} \frac{2}{p-1} \left( \frac{p-1}{2} n \log_p n \right)  \log p = \pi(x) \,n \log n. 
\end{equation} 
 Applying  inequality \eqref{eqn:S-ineq} for $S_p(n)$ term-by-term to the definition of $A(n,x)$ yields
  \begin{equation}\label{eqn:Anx-star-formula} 
A(n,x) \le A^{\ast}(n,x) =  \pi(x) n \log n.
\end{equation} 
Furthermore from the estimate \eqref{eqn:mirsky}  applied term-by-term  to the definition of $A(n,x)$, 
we obtain, viewing $x$ as fixed and $n$ as varying. 
\begin{equation}\label{eqn:Anx-O-estimate}
A(n,x) = \pi(x) n \log n +O_{x}(n),  
\end{equation}
where the implied constant in the $O$-symbol depends on $x$. 
 It follows that  for fixed $x$ 
 one has the  the asymptotic formula 
\begin{equation}\label{eqn:Anx-asymptotic-n}
A(n, x) \sim A^{\ast}(n,x) = \pi(x) \,n \log n \quad \mbox{as} \quad n \to \infty, 
\end{equation} 
Thus $A(n,x)$  encodes  information about $\pi(x)$ for  $n$ very large compared to $x$.
 In  the case where 
 $x=n$  \eqref{eqn:Anx-star-formula} gives 
 \begin{equation}
 A(n) = A(n,n) \le A^{\ast}(n,n) = \pi(n) n \log n.
 \end{equation} 
The   prime number theorem estimate  $\pi(n) = \frac{n}{\log n} +O \left( \frac{n}{(\log n)^2}\right)$ yields
$$
\pi(n) n \,\log n  = n^2 +O\left(\frac{n^2}{\log n} \right).
$$

The question of whether $A(n) \sim A^{\ast}(n,n)$ could hold as $n \to \infty$,
 was raised in \cite[Sect. 8]{LagM:2016}. 
By the prime number theorem it is  equivalent to the question whether $A(n) \sim n^2$  as $n \to \infty$. 
Theorem \ref{thm:ABn} answers this question in the negative,  showing that
 $A(n)\sim \left(\frac{3}{2} -\gamma \right)n^2$, with 
 $\frac{3}{2} - \gamma \approx 0.92288$. 
    
\subsubsection{Digit sums $d_b(n)$}    
    The digit sums  $d_b(n)$ are oscillatory quantities that have been modeled probabilistically, where one 
 samples for a fixed $b$, the values $d_b(k)$ uniformly in a certain range  of $k$. One 
  has for each $n \ge 1$ the inequality
 $$
 \EE[ d_b(k): 1 \le k \le n-1] = \frac{1}{n-1} S_b(n),
 $$
 and it follows that 
 \begin{equation}
   \EE[ d_b(k): 0 \le k \le n-1]  =  \frac{b-1}{2} \log_b n + O_b(1), 
 \end{equation} 
 according to \eqref{eqn:mirsky}. 
 Furthermore the bound \eqref{eqn:S-ineq}  gives
 \begin{equation}\label{eqn:dbx-bound} 
  \EE[ d_b(k): 0 \le k \le n-1]  \le  \frac{b-1}{2} \log_b n,
 \end{equation} 
 The statistic $B(n,x)$ averages over $\frac{n-1}{p-1} d_p(n) \log p$ holding $n$ fixed and varying $p$.
 Now \eqref{eqn:dbx-bound}  gives
$$
\frac{n-1}{b-1}  \EE[ d_b(k): 0 \le k \le n-1] \log b \le \frac{1}{2} \log  n.
$$ 
 If the averaging over $p$ in $d_p(n)$  in this statistic behaved similarly to  averaging over $n$  for fixed $n$,
 then we might expect $B(n,x)$ to behave similarly to the statistic 
 \begin{equation}\label{eqn:Bnx-star-formula} 
  B^{\ast}(n, x) :=   \sum_{p \le x} \frac{n-1}{p-1} \left( \frac{p-1}{2} \log_p(n)\right) \log p  = \frac{1}{2} \pi (x) n \log n. 
\end{equation}
The  prime number theorem  yields the estimate
$$
B^{\ast} (n, n) =  \frac{1}{2} n^2 + O \left(\frac{n^2}{\log n} \right).
$$
The  question whether $B(n) \sim B^{\ast}(n,n)$ holds as $n \to \infty$ is equivalent to whether $B(n) \sim \frac{1}{2}n^2$ as $n \to \infty$ holds.
Theorem \ref{thm:ABn} answers this question in the negative, with $B(n) = (1-\gamma) n^2$ as $n \to \infty$
and  $1- \gamma \approx 0.42288$.

\subsubsection{Asymptotics for $A(n,x)$ and $B(n,x)$ with $x=o(n)$}
The estimates for $A(n,n)$ and $B(n,n)$
 reveal  a difficulty in deducing the prime number theorem from
 radix expansion statistics, purely from  knowing the limiting statistics as $n \to \infty$ holding $p$ fixed.
Theorem \ref{thm:ABn}  shows that the 
 problem is that the contributions of individual primes $p$
 in these radix expansion statistics  have not reached their individual  limiting  asymptotics  as $n\to \infty$,
  holding $p$ fixed.    In addition, when  $x=\alpha n $ and $n \to \infty$,  the formulas for $f_A(\alpha)$ and $f_{B}(\alpha)$  
   exhibit oscillations in the main terms of the estimates for $A(n,x)$ and $B(n,x)$.
  
 In contrast we  show  that
 for certain ranges of
relatively large  $x=o(n)$ the  asymptotic formula
  $A(n,x) \sim A^{\ast}(n,x)$ is valid. 
 
\begin{thm}\label{thm:correct-asymp}  

 Suppose that a sequence  $(n_j, x_j)$ with $1 \le x_j \le n_j$  having   $n_j \to \infty$ as $j \to \infty$ satisfies the two conditions
\begin{equation}\label{eqn:hypotheses-asymp} 
\lim_{j \to \infty} \frac{x_j}{n_j} = 0 \quad \mbox{and} \quad \lim_{j \to \infty} \frac{ \log x_j}{\log n_j} =1.
\end{equation} 
Then,
\begin{equation}\label{eqn:Anx-star-asymptotic-0}
A(n_j, x_j) \sim A^{\ast}(n_j, x_j) := \pi(x_j) n_j \log n_j \quad \mbox{as} \quad j \to \infty,
\end{equation}
and  
\begin{equation}\label{eqn:Bnx-star-asymptotic-0}
\,\, \, B(n_j, x_j) \sim B^{\ast}(n_j, x_j) := \frac{1}{2} \pi(x_j) n_j \log n_j \quad \mbox{as} \quad j\to \infty.
\end{equation}
In consequence 
 \begin{equation}\label{eqn:Gnx-asymptotic-0}
\log G(n_j, x_j)  \sim \frac{1}{2} \pi(x_j) n_j \log n_j \quad \mbox{as} \quad j \to \infty
\end{equation} 
\end{thm}

Theorem \ref{thm:correct-asymp} is proved in Section \ref{subsec:ABnx}.
The  asymptotic formulas of Theorem \ref{thm:correct-asymp}  fail to hold.
for  values of $x$ smaller than \eqref{eqn:hypotheses-asymp} 
relative to $n$ 
 For example,  taking $x_j=n_j^{\alpha}$ for any fixed $\alpha$ with $2/3 < \alpha <1$ 
the right side of \eqref{eqn:Anx-star-asymptotic-0}  is  $A^{\ast}(n_j, x_j) \sim \frac{1}{\alpha} n_j x_j$ but 
Theorem \ref{thm:A2}  combined with the prime number theorem $\vartheta(x) \sim x$
shows the left side   of \eqref{eqn:Anx-star-asymptotic-0}     is $A(n_j, x_j) \sim n_j x_j$ in this range. Also $B^{\ast}(n_j, x_j) \sim \frac{1}{\alpha} n_j x_j$
while Theorem \ref{thm:A1} and the prime number theorem show $B(n_j, x_j) \sim \frac{1}{2 \alpha} n_j x_j$.

%
%
\subsection{Related work}\label{sec:16}

Binomial coefficients and their factorizations have been studied  
in prime number theory and in sieve methods.  In 1932 in one of his first papers
Erd\H{o}s \cite{Erd:32} used the central binomial coefficients ${{2n}\choose{n}}$ to get an elegant proof of 
 Bertand's postulate, asserting  that there exists a prime between $n$ and $2n$,
as well as   Chebyshev type estimates for $\pi(x)$ (\cite{Che:1852}). 
Later Erd\H{o}s   showed with Kalmar in 1937 
that such an approach could in principle yield the
 prime number theorem,  in the sense that suitable (multiplicative) linear combinations of factorials exist to
 give a sharper sequence of inequalities yielding the result.  However their  proof of the existence of such identities 
 assumed the prime number theorem to be true. The  proof with Kalmar was lost, but in 1980
Diamond and Erd\H{o}s \cite{DiaErd:80} reconstructed a proof. 
  For Erd\H{o}s's  remarks on the work with Kalmar see \cite[pp. 58--59]{Erd:97} and  Rusza \cite[Section 1]{Ruz:99}. 
 We mention  also that  the  internal structure of prime factors of the middle binomial coefficient ${{2n}\choose{n}}$ 
 has received detailed study, see Erd\H{o}s et al \cite{ErdGRS:75} and Pomerance \cite{Pom:05}.

An  earlier  paper of the second author and Mehta \cite{LagM:2016} studied  
products of unreduced Farey fractions, and in it expressed $\log \G_n$ 
 in terms of   radix digit  statistics $A(n)$ and $B(n)$.
 Another paper  \cite{LagM:2017} studied parallel questions for products of Farey fractions, which were  related to questions in prime number theory.
On digit sums $S_b(n)$,  a  formula of Trollope \cite{Tro:68}  found in 1968 for base $2$ led to  notable work of 
Delange \cite{Del:1975}, giving an exact formula for $S_b(n)$ for all $b \ge 1$.  
It asserts that, for a general base $b \ge 2$, 
\begin{equation}\label{eqn:delange} 
S_b(n) = \frac{b-1}{2} n \log_b(n) + f_b(\log_b n) n
\end{equation} 
where $f_b(x)$ is a continuous function, periodic of period $1$,
which is everywhere non-differentiable.  Substituting  $n=1$ gives $f_b(0)=0$,
and the inequality \eqref{eqn:S-ineq} 
implies that $f_b(x) \le 0$ for all real $x$.
Further work  on $S_b(n)$ includes Flajolet et al \cite{FGKPT94} and Grabner and Hwang \cite{GH05},  
discussed  in  a survey of  Drmota and Grabner \cite{DrmGra10}. 
For work on  the distribution of digit sums $d_b(n)$, see the survey of   Chen et al \cite{CHZ:2014}.

 Up  to now  direct information on  sums over radix expansions like $A(n)$ or $B(n)$ has not been  not been successfully used  
 to obtain  proofs of the prime number theorem. 
  The appearance of Euler's constant in their  asymptotics connects to many problems in number theory, cf. \cite{Lag:13}.
 The prime number theorem  has been successfully deduced 
 by elementary methods.  In 1945 Ingham \cite{Ing45}  deduced the prime number theorem  
 from a Tauberian theorem starting from asymptotic estimates of 
$$F(x) = \sum_{n \le x} f\left(\frac{x}{n}\right),$$
 under the Tauberian condition that $f(x)$ is positive and increasing.  The prime number theorem  was deduced 
from estimates of $\log n!$, by N. Levinson \cite{Lev64} in 1964  by a related method. These methods  obtain a
remainder term saving at most one logarithm. 
In 1970  Diamond and Steinig \cite{DiaSt:70} obtained by elementary methods a proof of the prime number theorem
  with a remainder term $O \large( x \exp( - c (\log x)^{\beta}) \large)$ for $\beta= \frac{1}{7} + \epsilon$.
  The exponent was improved to $\beta=\frac{1}{6} - \epsilon$
  by Lavrik and Sobirov \cite{LS73}.  In 1982 Diamond \cite{Dia82} gave a useful survey of such 
 approaches to the prime number theorem. 

%
%
\subsection{Contents of paper}\label{sec:17} 

 Section \ref{sec:2} derives estimates of $A(n)$ and $B(n)$.
 Section \ref{sec:asymp-Anx} derives estimates of $A(n,x)$ and $B(n,x)$,
 and proves Theorem \ref{thm:correct-asymp}. 
 In addition Theorem \ref{thm:ABnx}  in Section  \ref{subsec:ABnx} gives simplified formulas 
 for the main terms in the asymptotics of
 $A(n,x)$ and $B(n,x)$ which apply when $x= o(n)$.
Section \ref{sec:Gnx} derives estimates of $\log G(n,x)$. and proves
properties of the limit functions $f_{G}(\alpha)$. Section \ref{sec:5}
  determines the limit function $f_{BC}(\alpha)$
 for partial factorizations of the central binomial coefficients ${{2n}\choose{n}}$.
 Appendix \ref{sec:appendix1} presents  estimates for 
 $B(n,x)$  based on exponential sums due to O. Bordell\'{e}s, 
 yielding  improved estimations for $A(n,x)$, $B(n,x)$, $\log G(n,x)$
 for some ranges of  $x=o(n)$.

\subsection*{Acknowledgments}
We thank Olivier Bordell\`{e}s for
communicating the exponential sum estimates  given  in Theorem \ref{thm:A1}.
of the Appendix.
We are indebted to D. Harry Richman for providing plots of the limit functions,
and to Wijit Yangjit for helpful comments. 
We thank the reviewer for references and significant simplifications of proofs.
Theorem \ref{thm:ABn}  appears  in the PhD. thesis of the first author (\cite{Du:2020}),
who thanks  Trevor Wooley for helpful comments. 
The first author was partly supported by NSF grant DMS-1701577.
The second author was partly supported by NSF grants DMS-1401224 and DMS-1701576,
and  by a Simons  Fellowship in Mathematics in 2019.

 
%
%

  \section{Asymptotics for the sums  $B(n)$ and $A(n)$}\label{sec:2}

\par In this section we first  obtain  asymptotics for the functions   $B(n)=\sum_{p\leq n} \frac{n-1}{p-1}d_{p}(n)\log p$,
given in Theorem \ref{thm:ABn}(2). At the  the end we deduce asymptotics for $A(n) =\sum_{p\leq n} \frac{2}{p-1}S_{p}(n)\log p$.
  
  %
%

\subsection{Preliminary Reduction}\label{subsec:21}

We study $B(n)$ and reduce the main sum to primes in the range $\sqrt{n} < p \le n$. 
We write
\begin{equation}
B(n) = B_1(n) + B_R(n)
\end{equation}
where 
\begin{equation}\label{eq:B1}
B_1(n) :=\sum_{\sqrt{n}  <p\leq n } \frac{n-1}{p-1}d_{p}(n)\log p
\end{equation}
and 
\begin{equation}\label{eq:BR}
B_R(n) :=\sum_{1 <p\leq \sqrt{n}} \frac{n-1}{p-1}d_{p}(n)\log p.
\end{equation} 
is a remainder term coming from small $p$ (relative to $n$) that makes a negligible contribution to the asymptotics

%
%
\begin{lem}\label{lem:21}
 For $n \ge 2$ 
\begin{equation}
B_{R}(n) \le 4 \, n^{3/2} .
\end{equation}
\end{lem}

\begin{proof}
One has   
$d_p(n) \le (p-1) \left(\frac{\log n}{\log p} +1\right).$
Consequently
\begin{align*}
B_R(n) &
\le  \sum_{p \le \sqrt{n}} (n-1)(\log n + \log p)\\
&\le  \sum_{p \le \sqrt{n}} (n-1)(\log n + \log \sqrt{n})\\
&\le \frac{3}{2}n \pi(\sqrt{n}) \log n  \\
&\le 4\, n^{3/2},
\end{align*} 
The rightmost inequality used the  estimate,  valid  for  $x >1$, that 
\begin{equation} \label{eq:pix_upper} 
\pi(x) \le  1.25506\frac{x}{\log x},
\end{equation}
  see  Rosser and Schoenfeld \cite[eqn. (3.6)]{RosSch62}.
\end{proof}

%
%

\subsection{ Estimate for $B_{1}(n)$: radix expansion}\label{subsec:22}
\label{sec:B11asympt}

\par
 We   estimate  $B_1(n)$ starting from the observation that   for primes $ \sqrt{n} < p \le n $, 
 the base $p$ radix  expansion of $n$ for $ \sqrt{n} < p \le n $,
 has exactly $2$ digits. 
 
%
%
\begin{lem}\label{lem:diff1} 
 For $n \ge 2$  and all primes $\sqrt{n}< p\le n$, one has 
  $$d_p(n) = n-  (p-1)\left\lfloor \frac{n}{p} \right\rfloor.$$
In consequence  for all primes $ \sqrt{n}< p \le n$ lying in the  interval $I_j=(\frac{n}{j+1},\frac{n}{j}] $, where
$j = \lfloor \frac{n}{p}\rfloor$,   and $1 \le j < \sqrt{n}$,  one has 
\begin{equation}\label{eqn:digitsum}
\frac{n-1}{p-1} d_p(n) \log p = (n-1) \left(  \frac{n \log p }{p-1}-  j \log p\right).  
\end{equation} 
 \end{lem}

\begin{proof}

 For $\sqrt{n} <p< n$
the integer $n$ has  exactly two base $p$ digits, 
 $n  =  a_1 p+a_0$. Here  $a_1(n)=\lfloor \frac{n}{p} \rfloor$, corresponding to $p \in I_j=(\frac{n}{j+1},\frac{n}{j}] $, 
  the trailing digit $a_0(n) = n- p \lfloor \frac{n}{p} \rfloor  $ , whence
$$
d_p(n) = a_0(n) + a_1(n)= (n- p \lfloor \frac{n}{p} \rfloor ) + \lfloor \frac{n}{p} \rfloor =  n-  (p-1) \lfloor \frac{n}{p} \rfloor.
$$
This formula is a special case of  $d_p(n) = n- (p-1)\left( \sum_{k=1}^{\infty} \lfloor \frac{n}{p^k}\rfloor\right)$,
which  follows from computing $\nu_p(n!)$ two ways, the first being the Legendre formula $\nu_p(n!)= \sum_{k=1}^{\infty} \lfloor \frac{n}{p^k}\rfloor$
and the second being  $\nu_p(n!)  = \frac{n- d_p(n)}{p-1}$, 
 see Hasse \cite[Chap. 17, sect. 3]{Hasse80}. 
 
 Now the condition  $j= \lfloor \frac{n}{p} \rfloor$ corresponds to $p \in I_j =(\frac{n}{j+1},\frac{n}{j}]$, 
 and \eqref{eqn:digitsum} follows by substitution of the value of $d_p(n)$ when $p > \sqrt(n)$.
  Note  that   the intervals $I_j$ for $1 \le j < \sqrt{n}$ cover the entire interval $\sqrt{n}< p \le n$
  (and may include some $p \le \sqrt{n}$ in the last interval, where $n$ has three digits in its base $p$ radix expansion). 
\end{proof} 

We use the identity  \eqref{eqn:digitsum} to  split the sum $B_1(n)$ into two parts: 
\begin{equation}
B_1(n)= B_{11}(n)- B_{12}(n) ,
\end{equation}
in which 
\begin{equation}
\label{eqn:B11}
B_{11}(n) :=n(n-1)  \sum_{\sqrt{n} < p\leq n} \frac{\log p }{p-1}, 
\end{equation}
and
\begin{equation}
\label{eqn:B12}
B_{12}(n) := (n-1)\sum_{j=1}^{\lfloor  \sqrt{n} \,\rfloor } \, j\Bigg[\sideset{}{'}\sum_{\frac{n}{j+1}<p\leq \frac{n}{j}}  \log p\Bigg].
\end{equation}
where the prime in the inner sum means  only $p> \sqrt{n}$ are included. 
(The prime only affects one term in the sum.) 
The sums $B_{11}(n)$ and $B_{12}(n)$ are  of comparable sizes, on the order of  $n^2$.
We estimate them separately.

%
%

\subsection{ Estimate for $B_{11}(n)$}
\label{subsec:B21asympt}

The first quantity $B_{11}(n)$ is a standard sum in number theory.

%
%
\begin{thm}
 \label{thm:b1} 
  Let 
$$B_{11}(n) =n(n-1)  \sum_{\sqrt{n}< p\leq n} \frac{\log p }{p-1}.$$

(1) There is an absolute constant $c>0$ such that for $n \ge 4$, 
\begin{equation} 
\label{eqn:Uncond-B1-asymp}
B_{11}(n)=\frac{1}{2}n^{2}\log (n)+O(n^{2}e^{-c/2 \,\sqrt{\log n}}).
\end{equation} 
 
 (2) Assuming the Riemann hypothesis we have
 \begin{equation}
 \label{eqn:RH-B1-asymp}
 B_{11}(n) = \frac{1}{2}n^{2}\log (n)+O\left( n^{7/4} (\log n)^2\right).
 \end{equation}  
\end{thm}

We prove this result after a series of preliminary lemmas. 
As a first reduction, we show that $\frac{\log p }{p-1}$ may be approximated by $\frac{\log p}{p}$.
with a power savings error for $p > \sqrt{n}.$
%
%
\begin{lem}\label{lem:diff2} 
We have, unconditionally,  
\begin{equation}\label{eqn:primesum}
\sum_{\sqrt{n} < p\leq n} \frac{\log p}{p-1}= \sum_{\sqrt{n} < p\leq n} \frac{\log p}{p}+O\left(\frac{1}{\sqrt{n}} \right).
\end{equation}
 \end{lem}
 
\begin{proof}
We have 
\begin{eqnarray*} 
\sum_{\sqrt{n}< p\leq n} \left(\frac{\log p}{p-1}-\frac{\log p}{p}\right)
& = & \sum_{\sqrt{n}< p\leq n} \frac {\log p}{p(p-1)}
 \le  2\sum_{\sqrt{n}< p\leq n} \frac {\log p}{p^{2}}=O\left(\frac{1}{\sqrt{n}} \right)\\
\end{eqnarray*}
as required.
\end{proof} 

\par 
To estimate the sum on the right side of \eqref{eqn:primesum},
we study  $h(n) := \sum_{p\leq n} \frac{\log p}{p}$.
Merten's first theorem says that  the function 
$h(n)= \log n + O(1)$
 (see \cite[Theorem 425]{HW79}, \cite[Sect. I.4]{Ten15}). 
 Here  we need an estimate with a better remainder term. 

%
%

\begin{lem}
\label{lem:25} 

(1) There is a constant $c_2= \gamma- c_1$,
where $c_1= \sum_{p} \sum_{k=2}^{\infty} \frac{\log p}{p^k}$ such that, for $x \ge 4$, 
 $$
h(x) :=  \sum_{p\leq x}\frac{\log p}{p}  = \log x+ c_2 + 
 O\left(e^{-c\sqrt{\log x}}\right),
 $$
 
 (2) Assuming the Riemann hypothesis, for $x \ge 4$, 
  $$
h(x) :=   \sum_{p\leq x}\frac{\log p}{p} =
 \log x+   c_2 + O\left(x^{-1/2}(\log x)^2\right),
 $$
 \end{lem}

\begin{proof}
(1) This result appears in Rosser and Schoenfeld \cite[eqn. (2.31)]{RosSch62}. 

(2) This result appears in Schoenfeld \cite[eqn. (6.22)]{Sch76}.
\end{proof}

%
 %
\begin{defn}\label{def:26} 
(1) The {\em  first Chebyshev function } $\vartheta(n)$, is defined by 
$$\vartheta(n)=\sum_{p \leq n} \log p.  $$ 

(2) The {\em second Chebyshev function}  $\psi(n)$ is defined by
$$
\psi(n) = \sum_{{p,\,k}\atop{p^k \le n}} \log p = \sum_{m=1}^n \Lambda(m).
$$
\end{defn}

Here $\psi (n) = \vartheta(n) + \vartheta (n^{1/2}) + \vartheta( n^{1/3}) + \cdots.$  Using
the Chebyshev style estimate  $\vartheta(n) \le 5n$ given in \eqref{eq:pix_upper}, one has  
$$
 \vartheta(n) \le \psi(n) \le \vartheta(n) +  5 \sqrt{n} \log n.
$$

We recall known  bounds for $\vartheta(n)$.
%
%
\begin{lem}
  \label{lem:26}
{\rm (Chebyshev function estimates)}

(1) There is a constant $c>0$ such that, for $x \ge 4$, 
 $$
\vartheta(x) =  \sum_{p\leq x}\log p = x+ 
 O\left(xe^{-c\sqrt{\log x}}\right),
 $$
 
(2) Assuming the Riemann hypothesis, for $x \ge 4$, 
  $$
\vartheta(x) =   \sum_{p\leq x}\log p = x+ 
 O\left(\sqrt{x} (\log x)^2\right),
 $$
 \end{lem}
\begin{proof} 
(1) is given in \cite[Theorem 6.9]{MV07}. 

(2) is given in \cite[Theorem 13.1]{MV07}.
\end{proof}

\begin{proof}[Proof of Theorem \ref{thm:b1}.]
Recall $B_{11}(n) = n(n-1)\left(\sum_{\sqrt{n} < p\leq n} \frac{\log p}{p-1}\right)$.

(1) Applying estimate (1) of Lemma \ref{lem:25} with  $x= n$ and with $x= \sqrt{n}$, subtracting the latter  
cancels the constant $C_1$ and yields
$$
\sum_{\sqrt{n}< p \le n} \frac{\log p}{p} = \log n- \frac{1}{2}\log  n + O \left( e^{-c\sqrt{1/2\log n}} \right).
$$
Combining this bound with Lemma \ref{lem:diff2} yields
$$
\sum_{\sqrt{n}< p \le n} \frac{\log p}{p-1} = \frac{1}{2} \log n + O \left( e^{-\frac{c}{2} \sqrt{\log n}} \right).
$$
Multiplying by $n(n-1)$, we obtain  the bound \eqref{eqn:Uncond-B1-asymp}.

(2) Assuming  the Riemann hypothesis, we  proceed the same way as above, using the Riemann hypothesis estimate (2) of Lemma \ref{lem:25} in place of (1).
\end{proof}

\subsection{Estimates for $B_{12}(n)$}\label{subsec:B12asympt} 

We estimate $B_{12}(n)$   by rewriting  it  in terms of Chebyshev summatory functions, and using known estimates.

%
%
\begin{thm} \label{thm:b2}
Let 
$$B_{12}(n) :=  (n-1)\sum_{j=1}^{\lfloor\sqrt{n}\rfloor }j\Bigg[\sideset{}{'}\sum_{\frac{n}{j+1}<p\leq \frac{n}{j}}  \log p\Bigg].$$
 Then: 
 
 (1) There is an absolute  constant  $c>0$ such that for all $n \ge 4$, 
 \begin{equation}
  \label{eqn:uncond-B2-asymp}
 B_{12}(n) =\frac{1}{2}n^{2}\log n+(\gamma-1)n^{2}+O\left(n^{2} \,\log n\,e^{-c/2\, \sqrt{\log n}}\right)
\end{equation}

 (2) Assuming the Riemann hypothesis we have, for  all $n \ge 4$, 
 \begin{equation}
 \label{eqn:RH-B2-asymp}
 B_{12}(n) = \frac{1}{2}n^{2}\log (n)+ (\gamma-1) n^2 + O\left( n^{7/4} (\log n)^2 \right). 
 \end{equation} 
\end{thm}

\begin{proof} 

(1) We have
\begin{align*}
\frac{1}{n-1} B_{12}(n) 
&= \sum_{j=1}^{\lfloor \sqrt{n}\rfloor} j\Bigg[\sideset{}{'}\sum_{\frac{n}{j+1}<p\leq \frac{n}{j}}  \log p\Bigg] \\
&=\sum_{j=1}^{\lfloor \sqrt{n}\rfloor}j\left(\vartheta\left(\frac{n}{j}\right)-\vartheta\left(\frac{n}{j+1}\right)\right)+ O \left( \sqrt{n}\log n\right) \\
&=\left(\sum_{j=1}^{\lfloor \sqrt{n} \rfloor} \vartheta\left(\frac{n}{j}\right)\right)-(\sqrt{n}-1)\vartheta(\sqrt{n})+ O \left( \sqrt{n}\log n\right) 
\end{align*}
\noindent where $\vartheta(m)=\sum_{p \leq m} \log p$ is the first Chebyshev summatory function,
and the error estimate comes from not counting primes $p \le \sqrt{n}$ inside the term with $j \le  \sqrt{n}< j+1$. \\

Using Lemma \ref{lem:26},  for $j \le \sqrt{n}$ we have
$$
\vartheta \left(\frac{n}{j}\right) = \frac{n}{j} + O\left( \frac{n}{j} e^{-c \sqrt{ \log \frac{n}{j}} }\right) = \frac{n}{j} +  O\left(\frac{n}{j} e^{-\frac{c}{2} \sqrt{\log n} } \right).
$$
In consequence
\begin{align*}
\sum_{j=1}^{\lfloor \sqrt{n} \rfloor} \vartheta\left(\frac{n}{j} \right)
&=\sum_{j=1}^{\lfloor \sqrt{n} \rfloor}\left( \frac{n}{j} + O \left(  \frac{n}{j} e^{- \frac{c}{2} \sqrt{\log n}}  \right) \right)
- (\sqrt{n}-1)\left( \sqrt{n} + O\left( \sqrt{n}e^{- \frac{c}{2} \sqrt{\log n}} \right)\right) \\
&=n \left(\frac{1}{2}  n \log n + \gamma + O\left(\frac{1}{n}\right)\right) +  O \left( n \log n e^{- \frac{c}{2} \sqrt{\log n}} \right). 
\end{align*} 
In addition we have
$$(\sqrt{n}-1)\vartheta(\sqrt{n})= n + O \left(n \log n \,e^{- \frac{c}{2} \sqrt{\log n}}\right).$$
Substituting these formulas in the formula for $\frac{1}{n-1} B_{12}(n)$ above and multiplying by $n-1$ yields
\begin{align*} 
B_{12}(n) 
&= (n-1) \left( \frac{1}{2}  n \log n + \gamma n+  O \left(n \log n \,e^{- \frac{c}{2} \sqrt{\log n} } \right)\right) \\
& \quad\quad - (n-1) \left(n + O \left(n \log n e^{- \frac{c}{2} \sqrt{\log n}}\right)   \right)+O \left( n^{3/2} \log n \right) \\
&= \frac{1}{2} n^2 \log n + (\gamma - 1) n^2 + O\left( n^2 \log n \,e^{- \frac{c}{2} \sqrt{\log n}}\right),
\end{align*} 
as asserted.

(2) Now assume the Riemann hypothesis. Then 
\begin{align*}
\sum_{j=1}^{\lfloor \sqrt{n} \rfloor} \vartheta \left( \frac{n}{j} \right) 
&=\Bigg[ \sum_{j=1}^{\lfloor \sqrt{n} \rfloor} \frac{n}{j} + O \left( \sqrt{\frac{n}{j}} \left(\log \frac{n}{j}\right)^2 \right) \Bigg] \\
& = n \left( \frac{1}{2} \log n + \gamma + O\left(\frac{1}{n} \right)\right)  + O \left( n^{3/4} (\log n)^2 \right) 
\end{align*}

We also have  
$$(\sqrt{n}-1)\theta(\sqrt{n})=n+O\left(n^{3/4} (\log n)^2\right).$$
Consequently
\begin{align*}
B_{12}(n) 
&= (n-1) \left( \frac{1}{2}  n \log n + \gamma \, n +  O \left( n^{3/4}(\log n)^2 \right) \right)- (n-1) \left(n + O\left(n^{3/4} (\log n)^2 \right) \right) \\
&= \frac{1}{2}n^2 \log n + (\gamma-1) n^2 + O\left(n^{7/4} (\log n)^2 \right).
\end{align*}
\end{proof}

 %
%

 \subsection{Asymptotic estimate for $A(n)$ and $B(n)$}

\begin{proof}[Proof of Theorem \ref{thm:ABn}]
We estimate $B(n)$ and start with 
$$
B(n)= B_{11}(n)- B_{12}(n)+ B_R(n).
$$
By Lemma \ref{lem:diff1} we have $B_R(n) = O(n^{3/2})$,
which is negligible compared to the remainder terms in the theorem statements.\medskip

(1) Unconditionally, using  
Theorems \ref{thm:b1}(1) and Theorem \ref{thm:b2}(1), we obtain
\begin{align*}
B(n)&=B_{11}(n)-B_{12}(n)+ B_R(n)\\
&=\left(\frac{1}{2}n^{2}\log (n)+O\left(n^{2}e^{-\frac{c}{2}\sqrt{\log n}}\right)\right)-\left(\frac{1}{2}n^{2}\log n+(\gamma-1)n^{2}+O\left(n^{2}\log (n) e^{-\frac{c}{\sqrt{2}} \, \sqrt{\log n}}\right)\right)\\
&=(1-\gamma)n^{2}+O\left(n^{2}e^{-\frac{c}{2}\sqrt{\log n}}\right).
\end{align*}

(2) Assuming  the Riemann hypothesis, using
  Theorems \ref{thm:b1}(2) and Theorem \ref{thm:b2}(2), 
\begin{align*}
B(n)&= B_{11}(n)-B_{12}(n)+ B_R(n) \\
&=\left(\frac{1}{2}n^{2}\log (n)+O\left(n^{7/4} (\log n)^2\right)\right)  -\left(\frac{1}{2}n^{2}\log n+(\gamma-1)n^{2}+O\left( n^{7/4}(\log n)^2\right)\right)\\
&=(1-\gamma)n^{2}+O\left(n^{7/4}(\log n)^2\right),
\end{align*}
as required.

 The estimates for $A(n)$ follow directly from those of $B(n)$, using the linear relation
 $A(n) = \log \G_n + B(n)$. 
Combining this relation  with the asymptotic estimate \eqref{eqn:logG-asymp} yields 
$$
A(n) = \frac{1}{2} n^2  + B(n) + O \left(n \log n\right).
$$
The estimates (1) and (2) for $A(n)$ then follow on substituting the formulas (1), (2) for $B(n)$.
\end{proof} 
%
%

  \section{Asymptotic estimates  for the 
   sums  $B(n,x)$ and $A(n,x)$}\label{sec:asymp-Anx}

 %
%
\subsection{Estimates for $B(n,x)$}\label{subsec:Bnx}

We derive estimates for $B(n,x)$ in the interval $1 \le x  \le n$ 
starting from the asymptotic estimates for $B(n)= B(n,n)$.
Let  $H_m= \sum_{k=1}^m \frac{1}{k}$ denote the $m$-th harmonic number.

\begin{thm}\label{thm:Bnx}
Let $B(n,x) = \sum_{p\leq x}\frac{n-1}{p-1} d_p(n) \log p.$ 
We set
\begin{equation}\label{eqn:Bnx-bound}
B(n,x) =  B_0(n,x) + R_{B}(n,x),
\end{equation}
having   main term $B_0(n,x)= f_B(\frac{x}{n})n^2$ with
\begin{equation} \label{eqn:A0nx-def1}
f_{B}(\frac{x}{n}) := (1-\gamma)+ \left(  H_{\left\lfloor \frac{n}{x}\right\rfloor} -  \log \frac{n}{x} \right)
-  \left( \left\lfloor \frac{n}{x} \right\rfloor \frac{x}{n}\right),
\end{equation}
and  having $R_B(n,x)$ as  remainder term. Then: 

(1) Unconditionally for  all $n\ge 4$ and all 
$1 \le x  \le n$, the remainder term satisfies
\begin{equation}\label{eqn:BUcond}
R_{B}(n,x) =  O \left(n^2 \left(\frac{n}{x} \right) e^{- \frac{c}{2} \sqrt{ \log n}} \right),
\end{equation}
where the $O$-constant is absolute. 

(2) Assuming the Riemann hypothesis,   for $n^{3/4}\le x \le n$
the remainder term satisfies 
$$
R_{B}(n, x) = O \left( n^{7/4}  (\log n)^2 \right).
$$
\end{thm}

\begin{rem} \label{rem:32}
 It is immediate that the unconditional estimate (1)  is trivial
  whenever $1 \le x \le n \exp( -  c/2 \sqrt{ \log n})$,
 since the remainder term  will then have order of magnitude at least $n^2$,
 and the O-constant can be adjusted.
 The formula \eqref{eqn:Bnx-bound}  implies  a nontrivial estimate  for $x \ge \frac{n}{(\log n)^A}$ for any fixed positive $A$. 
\end{rem}

\begin{proof}
We  write
\begin{equation}
B(n, x) = B(n) - B^c(n,x),
\end{equation} 
where the complement function 
\begin{equation}\label{eqn:Bcnx}
B^c(n,x) := 
\sum_{x < p \le n} \frac{n-1}{p-1} d_p(n) \log p.
\end{equation}
 
 The analysis in Section \ref{sec:B11asympt} applies to estimate $B^c(n,x)$.
 We may assume that $x \ge \sqrt{n}$ since \eqref{eqn:BUcond} holds trivially for smaller $x$.
 For $x \ge \sqrt{n}$   Lemma \ref{lem:diff1} gives  the decomposition
$$
B^c(n,x) = B_{11}^c(n,x) - B_{12}^c(n,x)
$$
where
\begin{equation}\label{eqn:B11c} 
B_{11}^c(n,x) = n (n-1) \sum_{x < p \le n} \frac{\log p}{p-1}
\end{equation}
and
\begin{equation}\label{eqn:B12c}
B_{12}^c (n,x) = (n-1) \sum_{j=1}^{\lfloor \frac{n}{x}-1 \rfloor} j \left( \sum_{\frac{n}{j+1} < p \le \frac{n}{j}} \log p\right) 
 + (n-1) \left\lfloor \frac{n}{x} \right\rfloor \left( \sum_{x < p  \le  \frac{n}{\lfloor n/x \rfloor}} \log p \right).
\end{equation}

To estimate $B_{11}^c(n, x)$ we suppose   $x > \sqrt{n}$ and apply Lemma \ref{lem:diff2} to obtain
$$
\sum_{x < p \le n} \frac{\log p}{p-1} = \sum_{x < p \le n} \frac{\log p}{p} + O\left( \frac{\log n}{\sqrt{n}}\right).
$$
Next Lemma \ref{lem:25}(1) gives 
$$
\sum_{x < p \le n} \frac{\log p}{p} = \log \frac{n}{x} + O\left( \frac{ \log n}{\sqrt{n}}\right) .
$$
Substituting these two estimates in \eqref{eqn:B11c} yields,  for $x \ge \sqrt{n}$, unconditionally,  
\begin{equation}\label{eqn:B11cnx} 
B_{11}^c(n,x) = n(n-1) \log \frac{n}{x} + O\left( n^{3/2} \log n \right).
\end{equation}

To estimate $B_{12}^c (n, x)$ for  $x  \ge \sqrt{n}$, call the two sums on the right side of \eqref{eqn:B12c} $(n-1) B_{21}^c(n,x) $ 
and $(n-1)B_{22}^c(n,x)$, respectively. Then
\begin{eqnarray*} 
B_{21}^c(n,x)  &: =& \sum_{j=1}^{\lfloor \frac{n}{x}-1 \rfloor} j \left( \sum_{\frac{n}{j+1} < p \le \frac{n}{j}} \log p\right) 
= \sum_{j=1}^{ \left\lfloor \frac{n}{x} -1\right\rfloor} j \left(\vartheta \left(\frac{n}{j}\right) - \vartheta \left(\frac{n}{j+1}\right)\right)  \\
&=& \sum_{j=1}^{ \left\lfloor \frac{n}{x} -1\right\rfloor} j \left( \frac{n}{j} - \frac{n}{j+1} \right) +  O \left(\sum_{j=1}^{ \left\lfloor \frac{n}{x}  \right\rfloor } j \left( \frac{n}{j}\right) \exp( - c \sqrt{ \log (n/j)} ) \right)\\
&=& \sum_{j=1}^{\lfloor \frac{n}{x} -1\rfloor} \frac{n}{j+1} +  O \left(\left\lfloor \frac{n}{x} \right\rfloor n \exp( -c/2 \,  \log x)  \right)\\
&=& n \left( H_{ \left\lfloor \frac{n}{x} \right\rfloor} -1\right) + O \left(  \left\lfloor \frac{n}{x} \right\rfloor n \exp( - c/2\, \log x) \right),
\end{eqnarray*}
with  the prime number theorem with error term in Lemma \ref{lem:26}(1)  applied in the second line.
In addition
\begin{eqnarray*}
B_{22}^c (n,x) &:=& \left\lfloor \frac{n}{x} \right\rfloor \left( \sum_{x < p  \le  \frac{n}{\lfloor n/x \rfloor}} \log p \right)
 =   \left\lfloor \frac{n}{x} \right\rfloor \left(\vartheta \left( \frac{n}{\left\lfloor n/x \right\rfloor} \right) - \vartheta(x) \right) \\
&= & \left\lfloor \frac{n}{x} \right\rfloor \left(  \frac{n}{\left\lfloor n/x \right\rfloor}-x \right) + O \left(n \left\lfloor \frac{n}{x} \right\rfloor \exp( - c/2 \,  \sqrt{ \log x} )    \right) \\
&=&  n - \left\lfloor \frac{n}{x} \right\rfloor x + O \left(n \left\lfloor \frac{n}{x} \right\rfloor \exp( - c/2 \,  \sqrt{ \log n} )\right) ,
\end{eqnarray*}
also applying Lemma \ref{lem:26}(1)  in the second line.
Substituting  the bounds for $B_{11}^c(n,x)$ and $B_{22}^c(n,x)$ into \eqref{eqn:B11cnx} yields 
\begin{equation} 
B_{12}^c(n, x) = n(n-1) \left( H_{\left\lfloor \frac{n}{x}\right\rfloor} -  \left\lfloor \frac{n}{x} \right\rfloor \frac{x}{n} \right) + O \left( n^2 \left\lfloor \frac{n}{x} \right\rfloor \exp( - c/2 \,  \sqrt{ \log n} ) \right).
\end{equation}
We obtain, using  Theorem \ref{thm:ABn} (1)   to estimate $B(n)$, 
\begin{eqnarray}\label{eqn:Bnx-unconditional}
B(n, x) &= & 
 B(n)  -B_{11}^c(n, x) + B_{12}^c(n,x) \nonumber  \\
&=& (1-\gamma)n^2 + n(n-1) \left(-  \log \frac{n}{x} + H_{\left\lfloor \frac{n}{x} \right\rfloor} -  \left\lfloor \frac{n}{x} \right\rfloor \frac{x}{n} \right) \nonumber\\
&& \quad \quad\quad +O \left( n^2 \left\lfloor \frac{n}{x} \right\rfloor \exp( - c/2 \,  \sqrt{ \log n} )\right),\nonumber \\
&=& (1-\gamma)n^2 + n^2 \left( H_{\left\lfloor \frac{n}{x} \right\rfloor} -  \log \frac{n}{x}  \right)-  n^2\left( \left\lfloor \frac{n}{x} \right\rfloor \frac{x}{n} \right)\nonumber  \\
&& \quad \quad\quad +O \left( n^2 \left\lfloor \frac{n}{x} \right\rfloor \exp( - c/2 \,  \sqrt{ \log n} )\right), 
\end{eqnarray} 
which   is \eqref{eqn:Bnx-bound}.

(2) We follow the same sequence of estimates as in (1). 
In estimating both $S_1(n,x) $ and $S_2(n,x)$ we  apply the Riemann hypothesis bound in Lemma \ref{lem:26}(2) to improve
their  remainder terms from $O \left( n^{2} \left\lfloor \frac{n}{x} \right\rfloor \exp( - \frac{c}{2} \sqrt{ \log n})  \right)$
to $O\left( n^{3/2}  \left\lfloor \frac{n}{x} \right\rfloor (\log n)^2)\right)$. Imposing  the bound $ x \ge n^{3/4}$ yields the remainder term $O\left( n^{7/4} (\log n)^2\right)$. 
In the final sum \eqref{eqn:Bnx-unconditional} Theorem \ref{thm:ABn}(2) estimates $B(n)$ under the Riemann hypothesis to yield
 an additional  remainder term
$O \left( n^{7/4} (\log n)^2\right)$. 
\end{proof}

\begin{proof}[Proof of Theorem \ref{thm:Bnx-cor}] 
The theorem follows on choosing  $x= \alpha n$ in Theorem \ref{thm:Bnx} and simplifying. 
\end{proof}

\begin{rem}\label{rem:34}
The function $f_B(\alpha)$ defined by \eqref{eqn:Bnx-parametrized}  has  $f_B(1)=1-\gamma$, and has  $\lim_{\alpha \to 0} f_B(\alpha)=0$   since $ H_{\left\lfloor \frac{1}{\alpha} \right\rfloor}- \log \frac{1}{\alpha}  \to \gamma$
as $\alpha \to 0$. 
\end{rem}

%
%

  \subsection{Estimates  for  
   $A(n,x)$}\label{sec:32} 

We derive estimates for $A(n,x)$ starting from $A(x,x)$ and using a recursion involving $B(y, x)$
for $x \le y \le n$.

\begin{thm}\label{thm:Anx}
Let $A(n,x) = \sum_{p\leq x}\frac{2}{p-1} S_p(n) \log p.$  We write
\begin{equation}\label{eqn:Anx-bound}
A(n,x) =A_0(n,x)  + R_{A}(n,x), 
\end{equation}
having main term $A_{0}(n,x) = f_A(\frac{x}{n}) n^2$ with
\begin{equation}\label{eqn:A0nx-defn}
f_{A}(\frac{x}{n}) := \left(\frac{3}{2} -\gamma \right) +         \left(  H_{\lfloor \frac{n}{x} \rfloor} -  \log \frac{n}{x} \right) 
+ \frac{1}{2} (\frac{x}{n})^2 \left\lfloor \frac{n}{x}\right\rfloor^2 + \frac{1}{2} (\frac{x}{n})^2 \left\lfloor \frac{n}{x}\right\rfloor  - 2\frac{x}{n} \left\lfloor \frac{n}{x} \right\rfloor,
\end{equation} 
%
and having $R_{A}(n, x)$ as  remainder term. Then: 

(1) Unconditionally there  is a positive constant $c$ such that for  all $n \ge 4$ and $1 \le x \le n$, 
the remainder term satisfies
\begin{equation} \label{eqn:AUnd-rem} 
R_{A}(n,x)=   O \left(n^2 \left(\frac{n}{x} \right) e^{- \frac{c}{2} \sqrt{ \log n}} \right),
\end{equation}
where the $O$-constant is absolute.

(2) Assuming the Riemann hypothesis, for and $n \ge 4$ and $n^{3/4} \le x \le n$ the remainder term satisfies
$$
R_{A}(n,x) = O \left( n^{7/4}  \left(\frac{n}{x} \right) (\log n)^2 \right).
$$
\end{thm}

\begin{rem} \label{rem:36}
Although the range of $x$ in  (1) is given as $\sqrt{n} \le x \le n$, the 
 remainder term is larger than  the main term whenever $x \le n \exp( -  \frac{c}{2} \sqrt{ \log n})$.
 The formula \eqref{eqn:Anx-bound}  gives a nontrivial estimate  for $x \ge \frac{n}{(\log n)^A}$ for any fixed positive $A$.
\end{rem}


\begin{proof} 
(1) We start from  from the equality
$$
A(n,x) = \sum_{p \le x} \frac{2}{p-1} S_p(x) \log p  +   \sum_{y= x+1}^{n-1} \frac{2}{y-1} \left( \sum_{p \le x}  \frac{y-1}{p-1} d_p(y) \log p\right). 
$$
This formula  may be rewritten
\begin{equation}\label{eqn:AB-recursion}
A(n, x) = A(x, x) + \sum_{y= x+1}^{n-1}  \frac{2}{y-1} B(y, x). 
\end{equation} 
We apply the estimates of  Theorem \ref{thm:Anx-cor} (1) to $A(x,x)$, and those
of Theorem \ref{thm:Bnx} (1)   to $B(y, x)$, 
to obtain, for $x \ge \sqrt{n}$, 
\begin{eqnarray}\label{eqn:Anx-est-0}
A(n,x) & = & 
 \left(\frac{3}{2} - \gamma \right) x^2 
+ \sum_{y=x+1}^{n-1}  \frac{2}{y-1}\left( (1-\gamma)y^2+  y(y-1)\left(   H_{\left\lfloor \frac{y}{x} \right\rfloor} -\log \frac{y}{x}-
\left\lfloor \frac{y}{x}\right\rfloor \frac{x}{y}\right) \right)  \nonumber\\
&& \quad \quad  +O\left( x^2 \exp(- c \sqrt{ \log x} )\right)+ O \left(\sum_{y=x+1}^{n-1}  \frac{2}{y-1} \left(  y^2 \left\lfloor\frac{y}{x} \right\rfloor \exp ( - c/2 \sqrt{\log y})   \right) \right) \nonumber \\
&=&  
\left(\frac{3}{2} - \gamma \right) x^2 +  \left(  2(1-\gamma) \sum_{y=x+1}^{n-1} \frac{y^2}{y-1} \right) 
+ \left( \sum_{y=x+1}^{n-1}   2yH_{\lfloor \frac{y}{x} \rfloor} -2y  \log \frac{y}{x} -2x\left\lfloor \frac{y}{x}\right\rfloor \right) \nonumber\\
&&\quad\quad 
+O \left( n^2 \left\lfloor \frac{n}{x} \right\rfloor \exp ( - c/2 \sqrt{\log x})   \right).
\end{eqnarray}
We name the last two sums on the  right side of \eqref{eqn:Anx-est-0}, as
$$
A(n,x) = \left(\frac{3}{2} -\gamma \right)x^2 + A_1(n,x) + A_2(n,x)+O \left( n^2  \left\lfloor \frac{n}{x} \right\rfloor \exp ( - c/2 \, \sqrt{\log x})   \right).
$$
We assert
\begin{equation}\label{eqn:A1-est}
A_1(n,x) :=   2(1-\gamma) \sum_{y=x+1}^{n-1} \frac{y^2}{y-1}  = ( 1-  \gamma)(n^2- x^2) + O \left(n \right). 
\end{equation} 
This estimate follows from 
$$ \sum_{y=x+1}^{n-1} \frac{y^2}{y-1} = \sum_{y=x+1}^{n-1}( y+ O(1) )
 = \frac{1}{2} n(n-1) - \frac{1}{2} x(x+1) +  O(n)
 = \frac{1}{2}(n^2 - x^2) +O\left( n \right). 
$$
It remains to  estimate the sum
$$
A_2 (n, x) :=   \sum_{y=x+1}^{n-1} 2y H_{\left\lfloor \frac{y}{x} \right\rfloor} -  \sum_{y=x+1}^{n-1} 2y\left(\log \frac{y}{x}\right)- \sum_{y=x+1}^{n-1}  2x \left\lfloor \frac{y}{x} \right\rfloor    
= A_{21}(n,x)  - A_{22}(n,x)  - A_{23}(n,x).
$$
We assert that, for $n^{1/2} \le x \le n$,  
 \begin{equation}\label{eqn:A21-est} 
A_{21}(n,x)  =n^2 H_{\left\lfloor \frac{n}{x} \right\rfloor} - x^2 \left\lfloor \frac{n}{x} \right\rfloor^2 + x^2\left\lfloor \frac{n}{x} \right\rfloor  + O\left(  n \log n \right). 
\end{equation}
To show this, we evaluate the three  sums. We set   $n = j_0  x + \ell$  
$j_0 =\left\lfloor \frac{n}{x} \right\rfloor$ and $0 \le \ell < x$ , where
 $\ell = n- x \left\lfloor \frac{n}{x} \right\rfloor   =  x\{ \frac{n}{x} \} $.
 We have 
\begin{eqnarray*}
A_{21}(n,x) &=&  \sum_{y=x+1}^{n-1} 2y H_{\left\lfloor \frac{y}{x} \right\rfloor} = \sum_{j=1}^{j_0} \frac{1}{j} \left( \sum_{y=jx}^{n-1} 2y \right) -2x. \nonumber \\
&= & \sum_{j=1}^{j_0} \frac{2}{j} \left( {{n}\choose{2}}- {{jx+1}\choose{2}}\right) -2x \nonumber \\
&=&  n(n-1) H_{\left\lfloor \frac{n}{x} \right\rfloor} - \sum_{j=1}^{j_0} x(jx+1)   -2x \nonumber \\
&=& n(n-1) H_{\left\lfloor \frac{n}{x} \right\rfloor} - \frac{1}{2}x^2 \left\lfloor \frac{n}{x} \right\rfloor  \left\lfloor  \frac{n}{x}+ 1 \right\rfloor
 +x\left\lfloor \frac{n}{x} \right\rfloor -2x\nonumber, \\
\end{eqnarray*}
Now \eqref{eqn:A21-est} follows, since  the terms  
$x \left\lfloor \frac{n}{x} \right\rfloor -2 x$ and $n H_{\left\lfloor \frac{n}{x} \right\rfloor}$ contribute $O \left( n \log n\right)$.

We assert that
\begin{equation}\label{eqn:A22-est} 
A_{22}(n,x)  = n(n-1) \log \frac{n}{x} - \frac{1}{2}n^2 + \frac{1}{2} x^2 + O\left( n \log n\right).
\end{equation}
To see this, we have $A_{22} = 2( \sum_{y=x+1}^{n-1} y \log y) - 2( \sum_{y=x+1}^{n-1} y \log x).$ Now 
\begin{eqnarray*}
2 \sum_{y=x+1}^{n-1} y \log y &=&  2 \int_{x+1}^{n} y \log y \,dy + O \left( n \log n\right) =  \left(y^2 \log y - \frac{1}{2}y^2\right) |_{x+1}^{n} + O \left( n \log n\right) \\
&=& n^2 \log n - (x+1)^2 \log (x+1) -  \frac{1}{2} n^2 + \frac{1}{2}x^2 + O  \left( n \log n\right).
\end{eqnarray*}
We have also 
\begin{eqnarray*}
2\sum_{y=x+1}^{n-1} y \log x  &=& 
 n(n-1)\log x  -  x(x+1) \log x \\
&=& \left( n(n-1)\log n -  n(n-1) \log \frac{n}{x}\right) - (x+1)^2 \log (x+1) + O \left( n \log n \right).
\end{eqnarray*} 
Subtracting the last two estimates yields \eqref{eqn:A22-est}.

We assert that
 \begin{equation}\label{eqn:A23-est} 
A_{23}(n,x)  = -x^2 \left\lfloor \frac{n}{x}\right\rfloor^2 
%
- x^2\left\lfloor \frac{n}{x}\right\rfloor  + 2nx\left\lfloor \frac{n}{x} \right\rfloor +O\left(n \right).
\end{equation}
To see this, we have 
\begin{eqnarray*} 
A_{23}(n,x)  & = & \sum_{y=x+1}^{n-1}  2x \left\lfloor \frac{y}{x} \right\rfloor 
 =\left( 2x \left(\sum_{j=1}^{j_0-1} jx \right) -2x\right)  + 2 x  \left\lfloor \frac{n}{x} \right\rfloor \lfloor \ell +1\rfloor  \\
&=  &   x^2 \left(\left\lfloor \frac{n}{x}-1\right\rfloor  \right)\left(\left\lfloor \frac{n}{x}\right\rfloor \right) + 2 x\left\lfloor x \left\{ \frac{n}{x} \right\} \right\rfloor  \left\lfloor \frac{n}{x} \right\rfloor +O\left(n\right)  \nonumber. 
\end{eqnarray*}
%
%
We obtain \eqref{eqn:A23-est} by simplifying the last term on the right using
\begin{equation*} \label{eqn:fractional_part_simplify}  
\left\lfloor x \{ \frac{n}{x} \} \right\rfloor = x\left\{ \frac{n}{x} \right\} +O(1)= x\left( \frac{n}{x} - \left\lfloor \frac{n}{x} \right\rfloor \right)+O(1)  = n- x \left\lfloor \frac{n}{x} \right\rfloor +O(1). 
\end{equation*}

We insert the estimates  \eqref{eqn:A21-est} -\eqref{eqn:A23-est} 
into  $A_2(n,x)= A_{21}(n,x) - A_{22}(n,x) -A_{23}(n,x)$
(we replace coefficients $n(n-1)$  with $n^2$ modulo the remainder term), to obtain
\begin{equation}\label{eqn:A2-est}
A_2 (n, x)=   n^2 \left( H_{\left\lfloor \frac{n}{x} \right\rfloor} - \log \frac{n}{x}\right) 
+ \frac{1}{2} x^2 \left\lfloor \frac{n}{x} \right\rfloor^2   + \frac{1}{2} x^2 \left\lfloor \frac{n}{x} \right\rfloor       
-2nx \left\lfloor \frac{n}{x} \right\rfloor  + O \left( n \log n\right).
\end{equation}
Substituting the estimates \eqref{eqn:A1-est} and \eqref{eqn:A2-est} for $A_1(n, x)$ and $A_2(n,x)$ into  \eqref{eqn:Anx-est-0} 
yields \eqref{eqn:Anx-bound}.\smallskip

(2) Assuming the Riemann hypothesis, using the estimates of Theorem \ref{thm:ABn} for $B(n)$, and 
Theorem \ref{thm:Bnx}(2) for $B(y,x)$, the remainder term estimate for $A(n, x)$ given in \eqref{eqn:Anx-est-0} improves 
  to $O\left( n^{7/4} (\frac{n}{x}) (\log n)^2\right),$ for the range $n^{3/4} \le x \le n$. The reminder terms in all other estimates
are already $O (n \log n)$ so are absorbed in this remainder term. 
\end{proof}


\begin{proof}[Proof of Theorem \ref{thm:Anx-cor}] 
The result follows from Theorem \ref{thm:Anx} on choosing  $x= \alpha n$ and simplifying.
\end{proof}

\begin{rem}\label{rem:38}
The function $f_A(\alpha)$ defined by  \eqref{eqn:Anx-parametrized} has  $f_A(1)=\frac{3}{2} -\gamma$, and has  $\lim_{\alpha \to 0} f_A(\alpha)=0$   since 
$ H_{\lfloor \frac{1}{\alpha}\rfloor}- \log \frac{1}{\alpha}  \to \gamma$
as $\alpha \to 0$. 
\end{rem}

 %
%
\subsection{Simplified formulas  for main terms $A_0(n,x)$ and $B_0(n,x)$ when $x=o(n)$}\label{subsec:ABnx}

The main terms  $A_0(n,x)$ and $B_0(n,x)$  appearing in  Theorem \ref{thm:Anx} 
and Theorem \ref{thm:Bnx}   necessarily have  a complicated form, 
because they must describe the oscillations visible in the functions $f_A(\alpha)$ and $f_{B}(\alpha)$.
Here we  show their asymptotics simplify when $x=o(n)$,.

\begin{thm}\label{thm:ABnx} {\rm (Asymptotics of $A_0(n,x)$ and $B_0(n,x)$)} 

 (1)  Uniformly for  $n \ge 1$ and  all $1\le x\le n$,
\begin{equation}\label{eqn:A0x-asymp}
A_0(n,x) = nx + O(x^2). 
\end{equation}

(2) Uniformly for $n \ge 1$ and all $1 \le x \le n$,
\begin{equation}\label{eqn:B0x-asymp}
B_0(n,x) = \frac{1}{2}nx + O(x^2). 
\end{equation} 
\end{thm}

\begin{proof} We prove (2) and then (1).

(2)  Recall  $B_0(n,x) = f_{B}(\frac{x}{n}) n^2$ with
\begin{equation*}\label{eqn:B0nx-asymp}
f_B\left(\frac{x}{n}\right)  = \big(1-\gamma \big)+ \left(  H_{\left\lfloor \frac{n}{x}\right\rfloor} -  \log \frac{n}{x} \right) -  \left(  \left\lfloor \frac{n}{x} \right\rfloor  \frac{x}{n} \right) .
\end{equation*}
For  $t>1$,  we have 
\begin{equation} \label{eqn:harmonic-number-estimate}
H_{\lfloor t \rfloor} = \log \lfloor t \rfloor + \gamma + \frac{1}{2}\frac{1}{\lfloor t \rfloor} + O \left( \frac{1}{t^2} \right) 
\end{equation}
where $\gamma$ is Euler's constant, cf. \cite[eqn. (3.1.11)]{Lag:13}. 
(This estimate is valid  only at integer values $\lfloor t \rfloor$ because the remainder term is smaller than the jumps of
the step function at $\lfloor t \rfloor$.)  Taking $t = \frac{n}{x} \ge 3$, we obtain
\begin{eqnarray*}
f_{B}(\frac{x}{n} ) &= & \big(1-\gamma \big) +\left( \log \left\lfloor \frac{n}{x} \right\rfloor + \gamma +\frac{1}{2 \lfloor n/x\rfloor} + 
O\left(\frac{x^2}{n^2}\right) - \log \frac{n}{x} \right)
- \left\lfloor \frac{n}{x} \right\rfloor \frac{x}{n}. 
\end{eqnarray*}
Substituting $\frac {\lfloor t \rfloor}{t}  = 1- \frac{ \{ t \} }{t}, $ 
with $t= \frac{n}{x}$ the constant terms cancel and we obtain
\begin{equation}\label{eqn:simplified-2}
f_B(\frac{x}{n})  =\frac{1}{2 \lfloor n/x\rfloor}  + \log \left\lfloor \frac{n}{x} \right\rfloor  - \log \frac{n}{x}   +\frac{ \{n/x \}}{x/n} +  O\left(\frac{x^2}{n^2} \right)
\end{equation}
We next observe , for $t \ge 3$,
\begin{equation}\label{eqn:logt-frac}
 \log t - \log \lfloor t \rfloor = \log \left(\frac{t}{\lfloor t \rfloor } \right) = \log \left(1 + \frac{\{ t \} } {\lfloor t \rfloor}\right)  
 = \frac{ \{t \} }{\lfloor t \rfloor} +O \left( \frac{ \{t\}^2 }{ \lfloor t\rfloor^2} \right)= \frac{ \{t \} }{\lfloor t \rfloor} +O \left( \frac{1}{t^2} \right) 
 = \frac{ \{t \} }{ t} +O \left( \frac{1}{t^2} \right).  
 \end{equation}
Substituting this formula with   $t= \frac{n}{x}$ into \eqref{eqn:simplified-2}, the $\frac{x}{n}\{\frac{n}{x} \}$-terms cancel  and  we obtain 
  \begin{equation*}
f_{B}(\frac{x}{n} )  =   \frac{1}{2}\frac{1}{ \lfloor n/x \rfloor} +O \left( \frac{x^2}{n^2}  \right).
\end{equation*}
Using $\frac{1}{\lfloor t\rfloor} - \frac{1}{t} = O\left( \frac{1}{t^2} \right)$  (valid for $t \ge 1$) 
we obtain for $1 \le x \le \frac{1}{3}n$  that
$$
B_0(n,x) = f_{B}(\frac{x}{n}) x^2= \frac{1}{2} nx + O \left( x^2 \right).
$$
This estimate holds for the whole interval  $1 \le x \le n$, by increasing the $O$-constant to $1$ if it is smaller than $1$
since $B_0(n,x) \le B_0(n,n) \le (1-\gamma) n^2$.\smallskip

(1) Recall $A_0(n,x) = n^2 f_A(\frac{x}{n})$ with
\begin{equation*}\label{en:A0nx}
f_A\left(\frac{x}{n} \right) =  \left( \frac{3}{2} -\gamma \right)   +\left(  H_{\lfloor \frac{n}{x} \rfloor} -  \log \frac{n}{x} \right)
+ \frac{1}{2} \frac{x^2}{n^2}\left\lfloor \frac{n}{x}\right\rfloor^2 + \frac{1}{2} \frac{x^2}{n^2} \left\lfloor \frac{n}{x}\right\rfloor  - 2\frac{x}{n} \left\lfloor \frac{n}{x} \right\rfloor.
\end{equation*}
%
Taking $t = \frac{n}{x} \ge 3$, as in (1) we obtain
\begin{equation*}\label{eqn:simplified-4}
f_{A}(\frac{x}{n})  =  \left(\frac{3}{2}-\gamma \right) +\left( \log \left\lfloor \frac{n}{x} \right\rfloor + \gamma +\frac{1}{2 \lfloor n/x\rfloor} + 
O\left(\frac{x^2}{n^2}\right) - \log \frac{n}{x} \right)
+ \frac{1}{2} \frac{x^2}{n^2} \left\lfloor \frac{n}{x}\right\rfloor^2 + \frac{1}{2} \frac{x^2}{n^2} \left\lfloor \frac{n}{x}\right\rfloor  - 2\frac{x}{n} \left\lfloor \frac{n}{x} \right\rfloor.
\end{equation*}
We simplify the last expression by substituting  $ {\lfloor t \rfloor}  = t -  \{ t \} $ with $t = \frac{n}{x}$  to obtain
$$
 \frac{1}{2} \frac{x^2}{n^2} \left\lfloor \frac{n}{x}\right\rfloor^2 + \frac{1}{2} \frac{x^2}{n^2} \left\lfloor \frac{n}{x}\right\rfloor  - 2\frac{x}{n} \left\lfloor \frac{n}{x} \right\rfloor
 = \left(\frac{1}{2} - \frac{x}{n}  \left\{ \frac{n}{x} \right\}   + \frac{1}{2}\frac{x^2}{n^2}\left(\left\{ \frac{n}{x} \right\} \right)^2\right)
  +  \left( \frac{1}{2} \frac{x}{n} -\frac{1}{2} \frac{x^2}{n^2}  \left\{\frac{n}{x} \right\}\right) 
   - \left( 2- 2\frac{x}{n} \left\{ \frac{n}{x} \right\} \right). 
  $$
Substituting this formula  in
the  previous equation 
and using  \eqref{eqn:logt-frac} with $t = \frac{n}{x}$  we find the constant terms and the $\frac{x}{n} \{ \frac{n}{x} \}$-terms on the right side cancel, yielding for $1\le x \le \frac{1}{3} n$, 
\begin{eqnarray*}
f_A(\frac{x}{n})  &= &  \left( \frac{1}{2 \lfloor n/x\rfloor} + 
O\left(\frac{x^2}{n^2}\right)\right)
+\frac{1}{2}\frac{x^2}{n^2} \left(\left\{ \frac{n}{x}\right\} \right)^2  +  \left( \frac{1}{2} \frac{x}{n} -\frac{1}{2} \frac{x^2}{n^2}  \left\{\frac{n}{x} \right\}\right) \\
&=& \frac{x}{n} + O \left(\frac{x^2}{n^2} \right). 
\end{eqnarray*}
Multiplying by $n^2$ gives the result for $1 \le x \le \frac{1}{3}n$, and the estimate extends to $1\le x\le n$  similarly  to  (1),
(possibly changing the $O$-constant) using $A_0(n,x) \le A_0(n,n) =(3/2-\gamma)n^2$.
\end{proof} 

We apply  the simplified asymptotics of Theorem \ref{thm:ABnx} to prove Theorem \ref{thm:correct-asymp}. 

\begin{proof}[Proof of Theorem \ref{thm:correct-asymp} ]
The prime number
theorem together with the  hypothesis $\lim_{j \to \infty} \frac{ \log x_j}{\log n_j} =1$   implies
$$
\pi(x_j) \sim \frac{x_j}{\log x_j} \sim \frac{x_j} {\log n_j} \quad \mbox{as} \quad j \to \infty.
$$ 
We deduce
\begin{equation}\label{eqn:A-star-asymp-2} 
A^{\ast}(n_j, x_j) = \pi(x_j) n_j \log n_j \sim  n_j x_j \quad \mbox{as} \quad j \to \infty.
\end{equation}

For $A(n_j, x_j)$,  Theorem \ref{thm:Anx}(1)  gives
\begin{equation}\label{eqn:Anx-asymp-2} 
A( n_j, x_j) \sim  A_0(n_j, x_j)  \quad \mbox{as} \quad j \to \infty,
\end{equation}  
unconditionally  if $x_j \ge n_j \exp( - \frac{1}{2} c \sqrt{\log n_j})$ for all large enough $j$.
Now  Theorem \ref{thm:ABnx} gives 
$$
A_0(n_j, x_j) \sim x_j n_j
$$
over the entire range where $\frac{x_j}{n_j} \to 0$ and $\lim_{j \to \infty} \frac{\log x_j}{\log n_n} \to 1$. 

We use  Theorem \ref{thm:A2} for the remaining range of $x$ satisfying the hypothesis.
we get that for any sequence having  $\lim_{j \to \infty} \frac{x_j}{n_j} =0$ while  $x_j > (n_j)^{2/3}$ fo all large enough $j$,
we have
$$
A(n_j, x_j) \sim \vartheta(x_j) n_j  \quad \mbox{as} \quad j \to \infty.
$$
Now the prime number theorem implies $\vartheta(x) =\sim x$ as $x \to \infty$,
whence 
\begin{equation}\label{eqn:Anx-asymp-3} 
A(n_j, x_j) \sim x_j n_j \quad \mbox{as} \quad j \to \infty.
\end{equation} 
Combining \eqref{eqn:Anx-asymp-2} and \eqref{eqn:Anx-asymp-3} yields $A(n_j, x_j) \sim A^{\ast}(n_j, x_j)$
in the desired range of $x$.

The proof for $B(n_j, x_j) \sim B^{\ast}(n_j, x_j)$  is identical, using Theorem \ref{thm:Bnx},  Theorem \ref{thm:ABnx}(2), and Theorem \ref{thm:A1} 
in place of  Theorem \ref{thm:Anx} and Theorem \ref{thm:ABnx} (1) and Theorem \ref{thm:A2}. 
The identity  $\log G(n_j, x_j)= A(n_j, x_j) - B(n_j, x_j)$ then yields  the given asymptotic \eqref{eqn:Gnx-asymptotic-0} for $G(n_j, x_j)$. 
\end{proof}

%
%

 \section{Asymptotic estimates  for  $G(n,x)$}\label{sec:Gnx}

We deduce asymptotics of $G(n,x)$ and  study properties of its associated limit function $f_{G}(\alpha)$.

 %
%
\subsection{Estimates for $G(n,x)$}\label{subsec:Gnx}

%
%
\begin{thm}\label{thm:Gnx}
Let $G(n, x) = \prod_{p \le x} p^{\nu_p(\G_n)}$, and set
\begin{equation}\label{eqn:fnx}
f_G(\frac{x}{n}) = \frac{1}{2}  +\frac{1}{2} \left(\frac{x}{n}\right)^2 \left\lfloor \frac{n}{x}\right\rfloor^2 + \frac{1}{2} \left(\frac{x}{n}\right)^2 \left\lfloor \frac{n}{x}\right\rfloor 
- \frac{x}{n} \left\lfloor \frac{n}{x} \right\rfloor.
\end{equation}
for $0< \frac{x}{n} \le 1$.

(1) There is a constant $c >0$  such that for  all  $n \ge 4$ and $1 \le x \le n$, 
\begin{eqnarray}
\log G(n, x ) & = & f_G(\frac{x}{n}) \,n^2+ 
 O \left(n^2 \left(\frac{n}{x} \right) e^{- \frac{c}{2} \sqrt{ \log n}} \right),
\end{eqnarray}
where the implied $O$-constant is absolute.

(2) Assuming the Riemann hypothesis, for all $n \ge 4$ and $1 \le x \le n$,
\begin{equation}
\log G(n, x) =  f_G(\frac{x}{n}) \,n^2+
 O \left(  n^{7/4} \left( \frac{n}{x} \right) (\log n)^2 \right),
\end{equation}
The implied  $O$-constant is absolute. 
\end{thm} 

\begin{proof}
Recall from \eqref{eqn:GABx} the identity 
$$
\log G(n, x) = A(n, x) - B(n,x).
$$
The result (1) follows  by inserting  the formulas \eqref{eqn:Anx-bound} in Theorem \ref{thm:Anx} (1)  and \eqref{eqn:Bnx-bound} in Theorem \ref{thm:Bnx} (1)  into  the right side of this identity. 
The result (2) follows using the improved remainder terms in these formulas assuming the Riemann hypothesis.
\end{proof}

\begin{proof}[Proof of Theorem \ref{thm:Gnx-main}]
The theorem follows on choosing $x=\alpha n$ in Theorem \ref{thm:Gnx}, and simplifying.
Note that in the remainder term $\frac{n}{x}= \frac{1}{\alpha}$ appears to make 
the $O$-constant independent of $\alpha$.
\end{proof} 

%
%

 \subsection{Properties of limit function $f_G(\alpha)$}\label{subsec:fGa}

We establish properties of the  limit function $f_G(\alpha)$.

%
%
\begin{lem}\label{lem:fG} {\rm (Properties of $f_G(\alpha)$)}
Let $f_G(\alpha) = \frac{1}{2}   
+ \frac{1}{2} \alpha^2 \left\lfloor \frac{1}{\alpha}\right\rfloor^2 +
 \frac{1}{2} \alpha^2 \left\lfloor \frac{1}{\alpha} \right\rfloor- \alpha \left\lfloor \frac{1}{\alpha} \right\rfloor$.
 
(1) One has
\begin{equation}\label{eqn:fGaj}
f_G(\alpha) = \frac{1}{2} - j \alpha + \frac{1}{2} j(j+1) \alpha^2 \quad \mbox{for}  \quad \frac{1}{j+1} \le \alpha \le \frac{1}{j}.
\end{equation}

(2) The function $f_G(\alpha)$ is continuous on $[0,1]$, taking $f_G(0)=0$. One has $f_G(\frac{1}{j})= \frac{1}{2j}$ for $j \ge 1$.
 
 (3)  The function $f_G(\alpha)$ is not differentiable at $\alpha = \frac{1}{j}$ for $j \ge 2$,
nor at $\alpha=0$.

(4) One has 
\begin{equation}
f_G(\alpha) \le \frac{1}{2} \alpha \quad \mbox{for} \quad 0 \le \alpha \le 1.
\end{equation}
Equality occurs at $\alpha=0$ and at $\alpha = \frac{1}{j}$ for $j \ge 1$, and at no other point in $[0,1]$.
\end{lem} 

\begin{proof}
(1) Suppose $\frac{1}{j+1} < \alpha \le \frac{1}{j}$. Then $\lfloor \frac{1}{\alpha} \rfloor= j$, and $\{ \frac{1}{\alpha} \} = \frac{1}{\alpha} - j$.
Thus
\begin{eqnarray*}
f_G (\alpha) & = & \frac{1}{2} + \frac{1}{2} j^2 \alpha^2 + \frac{1}{2} j \alpha^2 - j\alpha\\
&=& \frac{1}{2} -j \alpha + \frac{1}{2}j(j+1)\alpha^2.
\end{eqnarray*}

(2) The quadratic function on the right side of \eqref{eqn:fGaj} has value $f_G( \frac{1}{j}) = \frac{1}{2j}$ and we check it 
continuously  extends to value $f_G(\frac{1}{j+1}) = \frac{1}{2(j+1)}$. The latter fact establishes continuity 
at the break point $\alpha = \frac{1}{j}$. On the half-open interval $(\frac{1}{j+1}, \frac{1}{j}]$ we have 
$$f'(\alpha) = -j + j(j+1)\alpha$$
which is positive on this interval, so $f(\alpha)$ is increasing on it.  Since $f(\frac{1}{j}) = \frac{1}{2j}$ we conclude $f(\alpha) \le \frac{1}{2j}$ for $0 < x \le \frac{1}{2j}$,
 hence   $\lim_{\alpha \to 0^{+}} f_G(\alpha) =0$.   Thus it is continuous at $\alpha=0$, on setting $f_G(0) =0$.
 
 (3)  At $\alpha=\frac{1}{j+1}$ the  derivative  approaching from the right is $0$ and approaching from the left is $1$. 
 Approaching $\alpha=0$ the derivative oscillates between $0$ and $1$ infinitely many times, and there is no limiting
 difference quotient approaching from the right.
 
 (4) Equality holds at $\alpha= \frac{1}{j+1}$ by property (2).
 On the interval $\frac{1}{j+1} \le \alpha \le \frac{1}{j}$ the quadratic
function is convex upwards, with initial slope $0$, and it touches the line $y= \frac{1}{2} x$ again at $x= \frac{1}{j}$. So the function must lie strictly
below the line $y = \frac{1}{2}x$ inside the interval. 
\end{proof} 
\section{Concluding Remarks}\label{sec:5}

This paper derived asymptotic information about the partial factorizations of products of binomial coefficients using
estimates from  prime number theory. 
 It showed  that   the functions   $A(n,x)$ and $B(n,x)$ related to partial factorizations   have
well-defined asymptotics  as $n \to \infty$, which under proper scalings when $s= \alpha n$ converge to limit functions,  
with  remainder terms  having  a power savings under the Riemann hypothesis. 

One would like to reverse the direction of information flow and derive
from such statistics  estimates
on  the distribution of prime numbers.  
To gain  insight  we  consider  the  simpler case of the central binomial coefficients ${{2n}\choose{n}}$, 
 where a rigorous result is possible.  We define
analogously  the partial factorizations 
\begin{equation}\label{eqn:51}
G_{BC}(2n, x) := \prod_{p \le x} p^{\nu_p( {{2n}\choose{n}})}. 
\end{equation} 
We have the Stirling's formula estimate
\begin{equation}\label{eqn:52}
 G_{BC}(2n,2n) =  {{2n}\choose{n}}= 4^{n + O(\log n)}.
\end{equation}  
Kummer's divisibility criterion implies that if $\sqrt{2n}< p < 2n$ then, for each $k \ge 1$,  
\begin{equation}\label{eqn:53}
\nu_p \left( {{2n}\choose{n}}\right)= \begin{cases}
1 \quad \mbox{if} \quad \frac{2n}{2k} < p \le \frac{2n}{2k-1}, \\
0 \quad \mbox{if} \quad \frac{2n}{2k+1} < p \le \frac{2n}{2k}.
\end{cases} 
\end{equation}
One may deduce in a fashion similar to the arguments in this paper that
\begin{equation}\label{eqn:58}
\log G_{BC}(2n, 2\alpha n) = f_{BC} (\alpha)2n + R_{BC}(2n, 2\alpha n), 
\end{equation}
 where $R_{BC}(2n, 2\alpha n)$ is a remainder term and
 $f_{BC}(\alpha)$ is a limit function defined  for $0 \le  \alpha \le 1$ 
 having   $f_{BC}(1) =\log 2\approx 0.69314$ and $f_{BC}(0)=0$ and 
 \begin{enumerate} 
   \item[(i)]
   $f_{BC}(\alpha)$ 
   is   continuous on $[0,1]$ and is piecewise linear  on $\alpha>0$.
   It is linear on intervals $[\frac{1}{k+1}, \frac{1}{k}]$ for $k \ge 1$.
  \item[(ii)] 
 $f_{BC}(\alpha)$ 
has   slope $1$ on intervals $\frac{1}{2k} \le \alpha \le \frac{1}{2k-1}$.
\item[(iii)]
$f_{BC}(\alpha)$ has 
  slope $0$ on intervals $\frac{1}{2k+1} \le \alpha \le \frac{1}{2k}$.
 \end{enumerate} 
 One can show using \eqref{eqn:53} that the reminder term   $R_{BC}(n, \alpha n)$ is  
 unconditionally of size $O\large(\frac{1}{\alpha}  n \exp( -c \sqrt{\log n})\large)$ and 
 is on the Riemann hypothesis  of size $O\left( \sqrt{\frac{1}{\alpha}}\,  n^{1/2} (\log n)^2\right)$.
 It is pictured in Figure \ref{fig:A51}.

%

\begin{figure}[h] 
\includegraphics[scale=0.60]{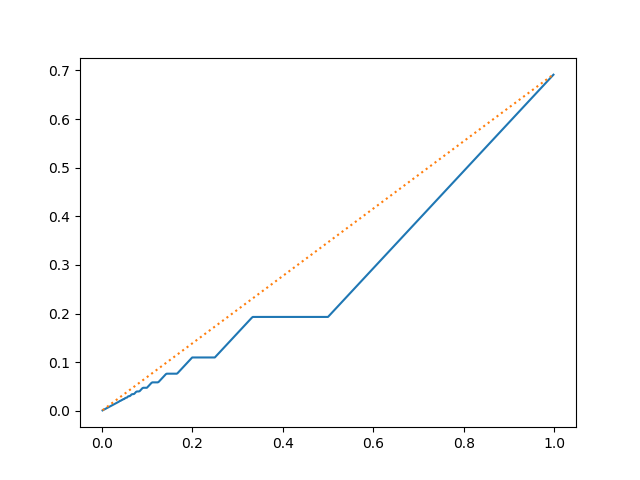}
\caption{Graph of limit function $f_{BC}(\alpha)$ in $(\alpha, \beta)$-plane,  $0 \le \alpha \le 1.$ The dotted line is $\beta= (\log 2 ) \alpha$.}
\label{fig:A51}
\end{figure} 

 The value $\alpha= \frac{1}{2}$  is especially interesting.
 It concerns $G(2n,n)$ and here Kummer's criterion gives
\begin{equation}\label{eqn:54}
  \frac{G_{BC}(2n,2n)}{G_{BC} (2n, n)}=\prod_{n < p \le 2n} p.
\end{equation} 
 It is well known
 that the Riemann hypothesis is equivalent to the assertion that for all integers $n \ge 2$, 
\begin{equation}\label{eqn:55} 
P(n) := \prod_{p \le n}p = e^{n + O(n^{1/2} (\log n)^2)}. 
\end{equation} 
(Taking  logarithms of \eqref{eqn:55},   $\log P(n)$ becomes Chebyshev's first function $\vartheta(n)$ 
and the equivalence follows from Lemma \ref{lem:26} (2).) 
In consequence one can  deduce \footnote{For the reverse direction estimate the logarithm of both sides of the telescoping product 
$$P(2n) = \prod_{j=0}^{\lfloor \log n\rfloor}  \frac{ P(2n/2^j)}{P(2n/2^{j+1})}.$$} 
that the Riemann hypothesis is also equivalent to the assertion that for all $n \ge 2$ ,
\begin{equation}\label{eqn:56} 
\frac{P(2n)}{P(n)} = \prod_{n < p \le 2n} p = e^{n + O (n^{1/2} (\log n)^2)}. 
\end{equation}  
We conclude that the  Riemann hypothesis is equivalent to the assertion that for all $n \ge 2$, 
the partial factorization $G(2n, x)$ with $x=n$ has 
\begin{equation}\label{eqn:57}  
G_{BC}(2n, n) =  G_{BC}(2n,2n) \cdot \frac{P(n)}{P(2n)} = \left( \frac{4}{e} \right)^{n +O (n^{1/2} (\log n)^2))}. 
\end{equation}
Taking logarithms in  \eqref{eqn:57}, we find that the Riemann hypothesis is equivalent to the assertion  that 
at $\alpha= \frac{1}{2}$, for all $n \ge 2$
\begin{equation}\label{eqn:59}
\log G_{BC}(2n,n) = 2 f_{BC} \left( \, \frac{1}{2}\, \right)  n + O\large( n^{1/2} (\log n)^2\large),
\end{equation}
with $f_{BC}(\frac{1}{2}) = \log 2 - \frac{1}{2} \approx 0.19314$. 
Thus the Riemann hypothesis is  encoded in a 
power-savings error term $O (n^{1/2} (\log n)^2)$
in \eqref{eqn:59} at the  {\em  single point}  $\alpha=1/2$.

 The role  the Riemann hypothesis  plays in these estimates 
 concerns the rapidity of convergence of the finite $n$ approximations  to these limit functions,
 and not in the  particular form of the limit function.   
 The central binomial coefficient  exhibits  a situation where   suitable power savings estimate  {\em at a single point}  $\alpha=\frac{1}{2}$   
 is equivalent  the Riemann hypothesis.

It may be   that the  power savings estimates given under RH for
binomial products in this paper for $0 < \alpha<1$ should  imply a zero-free region for the Riemann zeta function
of form $Re(s) > 1- \delta$ for some $\delta >0$.  We do not  know whether a power-savings  estimate at $\alpha= \frac{1}{2}$ alone 
would imply a zero-free region. 

This paper  started from  an expression  for $G(n, n)$ as a ratio of factorials, 
which led to a power savings estimate at $\alpha=1$.
The graph of the  limit function $f_G(\alpha)$ suggests that the values
$x =\frac{n}{j}$ might have special properties, since they lie on the line $y= \frac{1}{2}x.$
One may ask whether  factorial product formulas exist  for values $G(jn,n)$ when $j \ge 2$.

\appendix

\section{Estimates for $A(n,x)$,   $B(n,x)$ and $\log G(n,x)$ via exponential sums} \label{sec:appendix1} 

The following result  was communicated to us by Olivier Bordell\`{e}s.
One can obtain alternate unconditional  bounds for $B(n,x)$  by 
methods of exponential sums, having a  main term involving the first Chebyshev function $\vartheta(x)$,
which have nontrivial unconditional remainder terms in various  ranges where  $x=o(n)$.  
In this Appendix    $B_1(x)= x-\frac{1}{2}$ denotes  the first Bernoulli polynomial.

\begin{thm}\label{thm:A1} 
For $n \ge 1$  an  integer and $1 \le x \le n$ be a real number, set
\begin{equation}\label{eqn:A1}
B(n,x) = \frac{1}{2} \vartheta(x) n + \tR_{B}(n,x)
\end{equation}
where $\vartheta(x)$ is the first Chebyshev function
and $\tR_{B}(n,x)$ is the remainder.
Then for $n^{2/3} \le x \le n$, 
\begin{equation} \label{eqn:A2}
\tR_{B}(n,x)= O \left( x^{5/4} n^{3/4} (\log n)^{7/2}+n^{5/3} \log n \right).
\end{equation} 
\end{thm}  

\begin{proof}
Using Lemma  \ref{lem:21} and $\lfloor x\rfloor = x- \{ x\}$, 
we have
\begin{eqnarray*} 
B(n,x) &=& (n-1) \sum_{n^{1/2} <p \le x} \frac{\log p}{p-1} \left( n - (p-1) \lfloor \frac{n}{p} \rfloor\right) + O\left( n^{3/2} \right) \\
&=& n(n-1) \sum_{n^{1/2} <p \le x} \frac{\log p}{p-1} -    (n-1) \sum_{n^{1/2} <p \le x} \lfloor \frac{n}{p} \rfloor \log p + O \left( n^{3/2} \right)    \\
&=& \left( n(n-1) \sum_{n^{1/2} <p \le x} \frac{\log p}{p}  +    n(n-1) \sum_{n^{1/2} <p \le x} \frac{\log p}{p(p-1)}  \right)   \\
&& +\left( - n(n-1)\sum_{n^{1/2} <p \le x} \frac{\log p}{p} + (n-1)\sum_{n^{1/2} <p \le x}\{ \frac{n}{p} \} \log p \right) + O \left( n^{3/2} \right)\\
&=& n(n-1) \sum_{n^{1/2} <p \le x} \frac{\log p}{p(p-1)}+ (n-1) \sum_{n^{1/2} <p \le x} \{ \frac{n}{p} \} \log p + O \left( n^{3/2}  \right) .
\end{eqnarray*}  

The first Bernoulli polynomial has  $B_1(\{\frac{n}{p}\} ) = \{ \frac{n}{p} \} - \frac{1}{2}$, whence 
\begin{eqnarray*} 
B(n,x) &=&  n(n-1) \sum_{n^{1/2} <p \le x} \frac{\log p}{p(p-1)} + (n-1) \sum_{n^{1/2} <p \le x} B_1(\{\frac{n}{p}\}) \log p \\
&&+ \frac{n-1}{2}\left( \vartheta(x) - \vartheta (n^{1/2}) \right) 
                + O\left( n^{3/2} \right),\\
&=& \frac{1}{2} \vartheta(x) n + (n-1) \sum_{n^{1/2} <p \le x} B_1( \{ \frac{n}{p}\} ) \log p +O\left( n^{3/2} \log n \right). 
\end{eqnarray*}
We obtain 
\begin{equation} \label{eqn:lambda-estimate}
B(n,x) = \frac{1}{2} \vartheta(x) n + (n-1) \sum_{n^{1/2} < m \le x}\Lambda(m) B_1(\{\frac{n}{m}\})   +O\left( n^{3/2} \log n \right),
\end{equation} 
by inserting $O\left( \frac{n^{1/2}}{\log n} \right)$ extra nonzero terms $\Lambda(m)$  inside  the sum, each of size $O(\log n)$.

We  estimate the sum containing the von Mangoldt function.
If $ x \le 2n^{2/3}$ then 
$$|\sum_{n^{1/2} < m \le x}\Lambda(m) B_1(\{\frac{n}{m}\})| \le | \sum_{n^{1/2} < m \le 2n^{2/3}}\Lambda(m)| \ll n^{2/3},$$
which gives   \eqref{eqn:A2} with remainder term $\tR_{B}(n,x)=O (n^{5/3})$. 
For $2n^{2/3} \le x \le n$, 
 we have 
\begin{eqnarray}\label{eqn:dyadic-sum-0} 
|\sum_{n^{1/2} <m \le x} \Lambda(m) B_1(\{\frac{n}{m}\} )|  &\le &| \sum_{n^{1/2} <m \le 2n^{2/3}}\Lambda(m) B_1(\{\frac{n}{m}\})|  + |\sum_{2n^{2/3} <m \le x}\Lambda(m) B_1(\{\frac{n}{m}\})\, |  \nonumber \\
&\ll &n^{2/3} + \log n \left(\max_{2n^{2/3} < M \le x} | \sum_{M<m\le \min(2M, x) } \Lambda(m) B_1( \{ \frac{n}{m} \})\, |\right),
\end{eqnarray} 
We use the following estimate, cf. Graham and Kolesnik \cite[Theorem A6]{GK91}.
  For each integer $H \ge 1$ 
  and for $2n^{2/3}< M \le n$,
\begin{equation}\label{eqn:dyadic-sum-1}
|\sum_{M < m \le \min(2M,x)} \Lambda(m) B_1(\{\frac{n}{m}\})| \ll \frac{1}{H} \left(\sum_{M< m\le 2M} \Lambda (m)\right)  + \sum_{h=1}^H \frac{1}{h} | \sum_{M<m \le \min(2M,x)} \Lambda(m) \exp( \frac{2\pi i hn}{m})\, | ,
\end{equation}
with the  implied constant in $\ll$ being independent of $H$. (It  is based on  trigonometric polynomial majorants and minorants to 
the sawtooth function $B_1(\{x\})$.)  
We apply an  exponential sum estimate of  Granville and Ramar\`{e} \cite[Theorem 9', p.77]{GraR96}, which says: For $2n^{2/3} \le  M \le n$,
and any $y'$ with  $M \le y' \le 2M$, 
$$
| \,\sum_{M < m \le  y'} \Lambda(m) \exp( \frac{2 \pi i n}{m})\,| \le 5 M (\frac{M}{n})^{1/4} (\log 16M)^{5/2}. 
$$
Substituting this bound in \eqref{eqn:dyadic-sum-1}  (with $n$ replaced by $hn$ as needed) yields
for any integer $H \ge 1$ 
such that $ 2(Hn)^{2/3} \le M \le x$, 
\begin{eqnarray*}
|\sum_{M<m\le \min(2M,x)} \Lambda(m) B_1(\{ \frac{n}{m} \})\, | 
& \ll & \frac{M}{H} + \sum_{h=1}^H \frac{1}{h} \frac{M^{5/4}}{(hn)^{1/4}}(\log M)^{5/2} \\
&\ll& \frac{M}{H} + n^{-1/4} M^{5/4} (\log M)^{5/2} \\
&\ll& n M^{-1/2} + n^{-1/4} M^{5/4} (\log M)^{5/2},
\end{eqnarray*} 
where to get the  last line we choose $H= \lfloor \frac{1}{2}M^{3/2} n^{-1}\rfloor$.
Substituting these bounds into \eqref{eqn:dyadic-sum-0}, noting that $2n^{2/3} \le M \le x$, we obtain  for $2n^{2/3}\le x \le n$,
\begin{equation*}
| \sum_{n^{1/2}< m \le x} \Lambda(m) B_1(\{\frac{n}{m}\} ) |= O \left(    x^{5/4} n^{-1/4}(\log n)^{7/2} + n^{2/3}\log n \right).
\end{equation*} 
Substituting this estimate  in  \eqref{eqn:lambda-estimate} gives the desired bound for $\tR_{B}(n,x)$. 
 \end{proof} 

Following  the combinatorial approach in this paper, one
can deduce from Theorem \ref{thm:A1}
the following estimate for $A(n,x)$.

\begin{thm}\label{thm:A2} 
For  $n \ge 1$ an integer and $1 \le x \le n$  a real number,  set 
\begin{equation}\label{eqn:A3}
A(n,x) =  \vartheta(x) n + \tR_{A}(n,x)
\end{equation}
where $\vartheta(x)$ is the first Chebyshev function
and $\tR_{A}(n,x)$ is the remainder.
Then for  $ n^{2/3} \le x \le n$, 
\begin{equation} \label{eqn:A4}
\tR_{A}(n,x) = 
O \left( x^{5/4} n^{3/4} (\log n)^{7/2} + n^{5/3} (\log n)^2 \right).
\end{equation} 
\end{thm}  

\begin{proof} We use the combinatorial identity
$$
A(n,x) = A(x,x) + \sum_{y=x+1}^n \frac{2}{y-1} B(y,x).
$$
We have the trivial estimate $A(x, x) \le A^{\ast}(x, x) = x^2.$
Taking $n^{2/3} < x \le n$ and using the estimate of Theorem \ref{thm:A1},
we have
\begin{eqnarray*}
A(n, x) &= & O( x^2) + \sum_{y=x+1}^n \frac{2}{y-1}( \frac{1}{2}  \vartheta(x)y)  + O \left(\sum_{y=x+1}^n \frac{2}{y-1}( x^{5/4} y^{5/4} (\log y)^{7/2} + n^{5/3} \log n ) \right) \\
&=& 
\vartheta(x) (n-x)  +  O \left( \sum_{y=x+1}^n \frac{1}{y} \vartheta(x) \right) + O\left( x^{5/4} \sum_{y=x+1}^n y^{-1/4} (\log y)^{7/2} \right) + O \left( n^{5/3} (\log n)^2 \right) \\
&=& 
\vartheta(x) n  + O\left(x^2+  x \log n    + x^{5/4} n^{3/4} (\log n)^{7/2} + n^{5/3} (\log n)^2  \right)\\
&=& 
\vartheta(x) n + O\left(  x^{5/4} n^{3/4} (\log n)^{7/2} + n^{5/3} (\log n)^2 \right),
\end{eqnarray*} 
where the  last line takes the largest of the terms  in  the given range of $x$.
\end{proof}

\begin{cor}\label{cor:A3} 
For  $n \ge 1$  and  $1 \le x \le n$,   set 
\begin{equation}\label{eqn:A5}
\log G(n,x)  = \frac{1}{2} \vartheta(x) n + \tR_{G}(n,x)
\end{equation}
where $\vartheta(x)$ is the first Chebyshev function
and $\tR_{G}(n,x)$ is the remainder.
Then for  $n^{2/3} \le x \le n$, 
\begin{equation} \label{eqn:A6}
\tR_{G}(n,x) = O \left( x^{5/4} n^{3/4} (\log n)^{7/2} + n^{5/3} (\log n)^2 \right).
\end{equation} 
\end{cor}  

\begin{proof} 
Use the identity  $\log G(n,x) = A(n,x) - B(n,x)$ 
together with  the estimates in Theorem \ref{thm:A1} and Theorem \ref{thm:A2},
noting $\tR_{G}(n,x) = \tR_{A}(n,x)- \tR_{B}(n,x)$.
\end{proof} 

\begin{rem}\label{rem:A4} 
The estimates above have a nontrivial error term  for $x > n^{2/3} (\log n)^{2+ \epsilon}.$
O. Bordell\`{e}s also observes that one  can obtain   results parallel
to the Theorems above,  covering  the range $n^{1/2} <x \le  n^{2/3}$
having the same main terms and nontrivial
using an exponential sum estimate 
given in  Ma and Wu \cite[Proposition 3.1]{MaWu:20}.
\end{rem}

%
%


\begin{thebibliography}{99}

\bibitem{BelS:48}
R. Bellman and H. N. Shapiro,
A problem in additive number theory,
Annals of Math. {\bf 49} (1948), 333-340. 

\bibitem{Bush:40}
L. E. Bush, An asymptotic formula for the average sums of digits of integers,
Amer. Math. Monthly {\bf 47} (1940), 154--156.

\bibitem{Che:1852}
P. L. Chebyshev,
Memoire sur les nombres premiers,
J. Maths. Pures Appl. 1852, {\bf 17} 366--390.
[pp. 51--70 in: A. Markoff, N. Sonin, Editors, {\em Oeuvres de P. L. Tschebychef,}  Tome I,
St. Petersburg 1899.]

\bibitem{CHZ:2014}
L. H. Y. Chen, H-K Hwang, and V. Zacharovas,
Distribution of the sum of digits function of random integers: a survey,
Prob. Surveys {\bf 11} (2014), 177-236.

\bibitem{Del:1975}
H. Delange,
{Sur la fonction sommatoire de la fonction $\ll$Somme des chiffres $\gg$.}
L'Enseign. Math. {\bf 21} (1975), no. 1, 31--47.

\bibitem{Dia82}
H. Diamond,
Elementary  methods in the study of the distribution of prime numbers,
Bull. Amer. Math. Soc. (N.S.) {\bf 7} (1982), no. 3, 553--589.


\bibitem{DiaErd:80}
H. Diamond and P. Erd\H{o}s,
On sharp elementary prime number estimates,
Enseign. Math. {\bf 26} (1980), no. 3-4, 313--321.

\bibitem{DiaSt:70}
H. Diamond and J. Steinig,
An elementary proof of the prime number theorem with a remainder term,
Invent. Math. {\bf 11} (1970), 199--258.

\bibitem{DG52}
M. P. Drazin and J. S. Griffith,
{On the decimal representation of integers,}
Proc. Camb. Phil. Soc. {\bf 48} (1952), 555--565.

\bibitem{DrmGra10}
M. Drmota and P. J. Grabner,
Analysis of digital functions and applications,
pp.  452--504 in: {\em Combinatorics, automata and number theory,}
Encyclopedia Math. Appl. No. 135. Cambridge University Press, Cambridge  2010. 

\bibitem{Du:2020}
L. Du,
PhD thesis,
University of Michigan, 2020.

\bibitem{Erd:32}
P. Erd\H{o}s,
Bewies eines Satzes von Tschebischeff,
Acta. Litt. Sci. Szeged Sect. Math. {\bf 5} (1930/1932), 194--198.

\bibitem{ErdGRS:75}
P. Erd\H{o}s, R. L. Graham, I. Z. Rusza, E. G. Straus, On the prime factors of
${{2n}\choose{n}}$, Collection of articles in honor of Derrick Henry Lehmer on the occasion
of his seventieth birthday, Math. Comp {\bf 29} (1975), 83--92.

\bibitem{Erd:97}
P. Erd\H{o}s,
Some of my favorite problems and results,
pp. 47--67 in: {\em The Mathematics of Paul Erd\H{o}s}
(R. L. Graham and J. Neseteril, Eds.), Springer-Verlag, Berlin/New York 1997.

\bibitem{FGKPT94}
P. Flajolet, P. Grabner, P. Kirschenhofer, H. Prodinger and R. F. Tichy,
{Mellin transforms and asymptotics: digital sums.}
Theor. Comp. Sci. {\bf 123} (1994), 291--314.




\bibitem{GH05}
P. J. Grabner and Hsien-Kuei Hwang,
{Digital sums and divide-and-conquer recurrences: Fourier
expansions and absolute convergence,}
Const. Approx. {\bf 21} (2005), 149--179.




\bibitem{GK91}
S. W. Graham and G. Kolesnik,
\emph{Van der Corput's Method of Exponential Sums,}
Cambridge Univ. Press 1991. 


\bibitem{Gra97}
A. Granville,
{Arithmetic properties of binomial coefficients. I.  Binomial coefficients modulo prime powers.}
in: Organic mathematics (Burnaby, BC, 1995), 253--276, CMS Conf. Proc. 20,
Amer. Math. Soc. : Providence, RI  1997.

\bibitem{GraR96}
A. Granville and O. Ramar\'{e},
Explicit bounds on exponential sums and the scarcity of squarefree binomial coefficients,
Mathematika {\bf 43} (1996), 73--107.



 \bibitem{HW79}
G. H. Hardy, and E. M.  Wright, 
\emph{An Introduction to the Theory of Numbers (Fifth Edition).}
Oxford University Press: Oxford 1979. 

\bibitem{Hasse80}
H. Hasse,
{\em Number Theory.}
 Translated by H. G. Zimmer from the 1967 German edition.
Grundlehren der mathematischen Wissenschaften 229.
Springer-Verlag: Berlin-Heidelberg-New York 1980.
(Reprinted in series:  Classics in mathematics. Springer-Verlag: Berlin 2002.)


\bibitem{Ing45}
A. E. Ingham,
Some Tauberian theorems connected with the prime number theorem,
J. London Math. Soc. {\bf 22} (1945), 161--180.

\bibitem{Lag:13}
J. C. Lagarias,
{Euler's constant: Euler's work and modern developments.}
Bull. Amer. Math. Soc. (N. S.) {\bf  50} (2013), no. 4, 527--628.


\bibitem{LagM:2016} 
J. C. Lagarias and H. Mehta, 
{Products of binomial coefficients and unreduced Farey fractions}. 
International Journal of Number Theory, {\bf 12}  (2016), no. 1, 57-91.
 
 \bibitem{LagM:2017}
 J. C. Lagarias and H. Mehta, 
 {Products of Farey fractions}.
 Experimental Math.  {\bf  26 }( 2017) no. 1, 1--21.
 
 
 \bibitem{LS73}
 A. F. Lavrik and S. S. Sobirov,
 The remainder term in the elementary proof of the prime number theorem (Russian),
 Dokl. Akad. Nauk. SSSR {\bf 211} (1973), 534--536.

\bibitem{Lev64}
N. Levinson,
The prime number theorem from $\log n!$,
Proc. Amer. Math. Soc. {bf 15} (1964), 480--485.

\bibitem{MaWu:20}
J. Ma and J. Wu, 
On a sum involving the Mangoldt function,
Periodica Math. Hung., to appear.

 \bibitem{Mir:49}
 L. Mirsky,
 A theorem on representations of integers in the scale of $r$,
 Scripta Mathematica  {\bf 15} (1949), 11--12. 

\bibitem{MV07}
H. L Montgomery and R. C. Vaughan,
\textit{Multiplicative Number Theory I. Classical Theory,}
Cambridge University Press, Cambridge 2007.



\bibitem{Pom:05}
C. Pomerance,
Divisors of the middle binomial coefficient,
Amer. Math. Monthly {\bf 122} (2015)  636--644.

\bibitem{RosSch62}
J. B. Rosser and L. Sch\"{o}nfeld,
{Approximate formulas for some functions of prime numbers,}
Illinois J. Math. {\bf 6} (1962), 64-94


\bibitem{Ruz:99}
I. Rusza,
Erd\H{o}s and the integers,
J. Number Theory {\bf 79} (1999), 115--163.

\bibitem{Sch76}
L. Schoenfeld,
Sharper bounds for the Chebyshev functions $\theta(x)$ and $\psi(x)$. II,
Math. Comp. {\bf 30} (1976), 337--360. 

\bibitem{Ten15}
G. Tenenbaum,
\textit{Introduction to Analytic and Probabilistic Number Theory,}
Third Edition, American Math. Soc., Providence, RI 2015.

\bibitem{Tro:68}
J. R. Trollope,
{An explicit expression for binary digital sums.}
Math. Mag. {\bf 41} (1968), 21--25.

\end{thebibliography}
\end{document}